\documentclass[a4paper,11pt]{article}
\usepackage{authblk}
\usepackage{subfiles}
\usepackage{amsmath,amssymb,amsthm}
\usepackage[left=3cm, right=3cm, top=1in, bottom=1in]{geometry}
\usepackage{framed}
\usepackage{mdframed}
\usepackage{graphicx}
\usepackage{bm}
\usepackage{theoremref}
\usepackage{enumitem}
\usepackage{fancyhdr}
\usepackage{color}
\usepackage{float}
\usepackage{dsfont}
\usepackage{mathtools}
\usepackage{breqn}
\usepackage{bbm}
\usepackage{mathrsfs}
\usepackage{amstext}
\usepackage{hyperref}
\usepackage{accents}
\usepackage{tikz}
\usepackage{caption}
\usepackage{subcaption}
\usepackage{setspace}
\usepackage{pdfsync}
\usepackage{mathrsfs}
\usepackage{verbatim}
\usepackage[utf8]{inputenc}
\usepackage[english]{babel}
\usepackage{parskip}
\usepackage[titletoc]{appendix}

\newcommand{\notinsubfile}[1]{}

\DeclarePairedDelimiter\floor{\lfloor}{\rfloor}

\hypersetup{
	colorlinks   = true,
	citecolor    = blue,
	linkcolor    = blue
}

\numberwithin{equation}{section}
\newtheorem{theorem}{Theorem}[section]
\newtheorem{corollary}[theorem]{Corollary}
\newtheorem{lemma}[theorem]{Lemma}
\newtheorem{proposition}[theorem]{Proposition}

\theoremstyle{remark}
\newtheorem{definition}[theorem]{Definition}
\newtheorem{remark}[theorem]{Remark}

\let\oldproofname=\proofname
\renewcommand{\proofname}{\rm\bf{\oldproofname}}

\newcommand{\mcB}{\mathcal{B}}
\newcommand{\mcC}{\mathcal{C}}
\newcommand{\mcD}{\mathcal{D}}
\newcommand{\mcE}{\mathcal{E}}
\newcommand{\mcF}{\mathcal{F}}

\newcommand{\mcH}{\mathcal{H}}
\newcommand{\mcI}{\mathcal{I}}

\newcommand{\mcL}{\mathcal{L}}
\newcommand{\mcM}{\mathcal{M}}
\newcommand{\mcN}{\mathcal{N}}
\newcommand{\mcO}{\mathcal{O}}
\newcommand{\mcP}{\mathcal{P}}

\newcommand{\mcS}{\mathcal{S}}
\newcommand{\mcT}{\mathcal{T}}

\newcommand{\indic}{\mathds{1}}

\newcommand{\mbC}{\mathbb{C}}
\newcommand{\C}{\mathbb{C}}

\newcommand{\mbE}{\mathbb{E}}

\newcommand{\mbN}{\mathbb{N}}

\newcommand{\mbP}{\mathbb{P}}

\newcommand{\mbR}{\mathbb{R}}

\newcommand{\mbT}{\mathbb{T}}
\newcommand{\T}{\mathbb{T}}

\newcommand{\mbZ}{\mathbb{Z}}

\newcommand{\mfR}{\mathfrak{R}}

\newcommand{\mfm}{\mathfrak{m}}

\newcommand{\msD}{\mathscr{D}}

\newcommand{\msH}{\mathscr{H}}

\newcommand{\msS}{\mathscr{S}}

\newcommand{\supp}{\mathrm{supp}}
\newcommand{\eps}{\varepsilon}
\newcommand{\dd}{\mathop{}\!\mathrm{d}}
\newcommand{\id}{\mathrm{id}}

\newcommand{\tzero}{|_{t=0}}
\newcommand{\bigtzero}{\big|_{t=0}}

\newcommand{\vertiii}[1]{{\left\vert\kern-0.25ex\left\vert\kern-0.25ex\left\vert #1 
		\right\vert\kern-0.25ex\right\vert\kern-0.25ex\right\vert}}

\newcommand{\TV}{\mathrm{TV}}

\newlength{\myleftlen}

\usepackage{setspace}
\makeatletter
\def\@tocline#1#2#3#4#5#6#7{\relax
  \ifnum #1>\c@tocdepth 
  \else
    \par \addpenalty\@secpenalty\addvspace{#2}%
    \begingroup \hyphenpenalty\@M
    \@ifempty{#4}{%
      \@tempdima\csname r@tocindent\number#1\endcsname\relax
    }{%
      \@tempdima#4\relax
    }%
    \parindent\z@ \leftskip#3\relax \advance\leftskip\@tempdima\relax
    \rightskip\@pnumwidth plus4em \parfillskip-\@pnumwidth
    #5\leavevmode\hskip-\@tempdima
      \ifcase #1
       \or\or \hskip 1em \or \hskip 2em \else \hskip 3em \fi%
      #6\nobreak\relax
    \dotfill\hbox to\@pnumwidth{\@tocpagenum{#7}}\par
    \nobreak
    \endgroup
  \fi}
\makeatother

\title{A Stochastic Model of Chemorepulsion with Additive Noise and Nonlinear Sensitivity}

\author[1]{Ilya Chevyrev \thanks{\emph{Email:} ichevyrev@gmail.com}}
\author[2]{Ben Hambly \thanks{\emph{Email:} hambly@maths.ox.ac.uk}}
\author[3]{Avi Mayorcas \thanks{\emph{Email:} am3015@cam.ac.uk\\
\emph{Acknowledgments.} AM gratefully acknowledges support from the EPSRC Centre For Doctoral Training in Partial Differential Equations: Analysis and Applications [grant number EP/L015811/1] }}
\affil[1]{\small School of Mathematics, University of Edinburgh, Edinburgh EH9 3FD, UK}
\affil[2]{\small Mathematical Institute, University of Oxford, Oxford, OX2 6GG, UK}
\affil[3]{\small Isaac Newton Institute, Cambridge CB3 0EH, UK}

\date{\today}

\begin{document}

\maketitle

\begin{abstract}
    We consider a stochastic partial differential equation (SPDE) model for chemorepulsion, with non-linear sensitivity on the one-dimensional torus. By establishing an a priori estimate independent of the initial data, we
    show that there exists a pathwise unique, global solution to the SPDE. Furthermore, we show that the associated semi-group is Markov and possesses a unique invariant measure, supported on a H{\"o}lder--Besov space of positive regularity, which the solution law converges to exponentially fast.
    The a priori bound also allows us to establish tail estimates on the $L^p$ norm of the invariant measure which are heavier than Gaussian.
\end{abstract}

%
\section{Introduction}
We study the Cauchy problem and establish exponential ergodicity for the SPDE
\begin{equation}\label{eq:1dRepEquation}
\begin{cases}
\partial_t u -\partial_{xx} u = \chi\partial_x (u^2\partial_x \rho_u)+\xi, &\text{ on }\mbR_+\times \mbT,\\
-\partial_{xx}\rho_u = u-\bar{u}, &\text{ on } \mbT,\\
u\tzero = \zeta, &\text{ on }\mbT,
\end{cases}
\end{equation}
where $\xi$ is a space-time white noise, $\mbT:= \mbR/\mbZ$ is the one dimensional torus with unit volume and $\chi>0$ is a positive constant. The spatial average, $\bar{u}:=\langle u,1\rangle_{L^2(\mbT)}$, is equal to the $0^{\text{th}}$ Fourier mode and we will take $\zeta$ to be a H{\"o}lder--Besov distribution of possibly negative regularity. Without loss of generality we assume that the solution to the second equation of \eqref{eq:1dRepEquation} is mean free. 

Without noise, \eqref{eq:1dRepEquation} is a Keller--Segel model of chemorepulsion with nonlinear sensitivity and is an example from a wide family of parabolic PDE models of chemotaxis written, with some generality on the $d$-dimensional torus, $\mbT^d$, for $a\geq 0$, as
\begin{equation}\label{eq:DetermRepKS}
    \begin{cases}
    \partial_t u - \Delta u =  \nabla \cdot (f(u,\rho)\nabla \rho), & \text{ on } \mbR_+\times \mbT^d,\\
    a \partial_t \rho - \Delta \rho = u - \bar{u},&\text{ on }\mbR_+\times \mbT^d,\\
    u\bigtzero = u_0\geq 0,\quad \rho\bigtzero =\rho_0, & \text{ on }\mbT^d.
    \end{cases}
\end{equation}
Chemotaxis refers to the directed movement of cells or bacteria in response to a chemical field. In \eqref{eq:DetermRepKS},
 $u$ represents the cell or bacteria density, $\rho$ the strength of the chemical field and $f\in C^\infty(\mbR^2;\mbR)$ is a density dependent sensitivity function such that $f(0,y)=0$ for all $y\in \mbR$.
 The viscous terms are included to account for microscopic fluctuations in the cell and chemical diffusions.
 Chemorepulsion refers to a situation where cells are biased to move down the chemical gradient, modelled by choosing $f(x,y)\geq 0$ for $x\geq 0$, while chemoattraction refers to the situation 
 where cells are biased to move up the gradient, modelled by choosing $f(x,y)\leq 0$ for $x\geq 0$.
 In situations of typical interest, the chemical field will be generated by the cell population itself, which leads to the nonlinearity in \eqref{eq:DetermRepKS}.
 When $f(u,\rho)=\chi u$, for $\chi \in \mbR$, \eqref{eq:DetermRepKS} is referred to as the Patlak--Keller--Segel model and was first introduced by Keller and Segel in \cite{keller_segel_71}, building on earlier work by Patlak, \cite{patlak_53_randomWalk}.
 When $a =0$, \eqref{eq:DetermRepKS} is known as a parabolic-elliptic model and when $a>0$, as a parabolic-parabolic model.
 A common feature of chemotaxis models is strongly differing behaviour in the chemorepulsive and chemoattractive settings.
 For example, in the parabolic-elliptic Patlak--Keller--Segel model ($f(u,\rho)=\chi u$, $a=0$),
 global well-posedness holds in all dimensions in the chemorepulsive regime ($\chi>0$) and in both regimes (any $\chi \in \mbR$) when $d=1$.
 On the other hand, if $d\geq 2$, in the chemoattractive regime ($\chi<0$) there exists a finite threshold $-\infty<\bar{\chi}_d<0$ such that, under some additional conditions when $d>2$,
 the model exhibits finite time blow-up for all $\chi<\bar{\chi}_d$.
 We refer to \cite{hillen_painter_08_users,perthame_04} for a more detailed exposition of these phenomena and discussion of other models.
 For the one dimensional setting in particular see \cite{hillen_potapov_04,osaki_yagi_01} and for the chemorepulsive setting see \cite{cieslak_laurencot_morales_08}.

In this paper we study \eqref{eq:1dRepEquation}, a one-dimensional additive noise, stochastic version of \eqref{eq:DetermRepKS} in the chemorepulsive and parabolic--elliptic ($a=0$) regime, with nonlinear sensitivity function, $f(u,\rho)=\chi u^2$. Deterministic models including nonlinear sensitivity have been studied since the early work of Keller and Segel, see for example \cite{ford_lauffenburger_91,keller_segel_71_travelling,lapidus_schiller_76,stevens_othmer_97}. These works focused on modelling nonlinear sensitivity with respect to the chemical signal, i.e. $f(u,\rho)= \chi u \tilde{f}(\rho)$. More recently there has been interest in modelling nonlinear sensitivity with respect to the cell density, i.e. $f(u,\rho)= \tilde{f}(u)$. This interest has been motivated, in part, by desires to include volume filling, saturation and density dependent (quorum sensing) regularisation effects into chemotaxis models, \cite{hillen_painter_01,hillen_painter_02,horstmann_winkler_05,lai_xiao_17,tao_13,tao_winkler_12}. The nonlinearity that we treat in \eqref{eq:1dRepEquation} was already considered in the deterministic context in \cite{lai_xiao_17,tao_13}; in \cite{lai_xiao_17} it was shown that in the chemorepulsive setting the parabolic-elliptic model exhibits global existence and convergence to equilibrium when $\tilde{f}(u)\sim u^m$ for all $m>0$. In this case, the nonlinear sensitivity models an increased response to the chemorepulsant when cell density is high.

While it is common to study chemotaxis using PDE models, SPDE approaches can be used to model exogenous processes, study meta-stability or replicate physically relevant fluctuations around the continuum limit. In this work we focus on an additive noise model for two reasons.
Firstly, it is a relatively simple model for which we can accommodate space-time white noise and establish a full ergodic analysis.
Secondly, it allows us to make meaningful comparisons with additive noise stochastic reaction-diffusion equations which have been extensively studied; we refer to \cite{daPrato_debussche_03_quantization,moinat_weber_20_RD,mourrat_weber_xu_16,otto_weber_westdickenberg_14_invariant,tsatsoulis_weber_18} for an incomplete list of works concerning these models.
In particular, we highlight a comparison between our stochastic model, \eqref{eq:1dRepEquation}, and those considered in \cite{moinat_weber_20_RD}.
As in the above works, the main technical step we require is to obtain sufficiently strong a priori bounds on the pathwise solution which are independent of the initial data.
In our context we divide the solution into a regular and an irregular part and the key step is a careful analysis of the more regular, remainder equation exploiting the repulsive nature of our nonlinearity to compensate for lower order terms without definitive sign.
This is the content of Theorem \ref{th:1dRemainderAPriori}. Furthermore, we keep careful track of the exponents and constants through the proof, which in particular allows us
to obtain a tail bound, which is heavier than Gaussian, on the $L^p$ norm of the invariant measure. We give a detailed discussion on the implications of this bound and comparisons to the literature on stochastic reaction-diffusion equations in Remarks \ref{rem:PLimit}, \ref{rem:TailBound} and \ref{rem:Extensions}.

There are many natural questions that arise from our work. For this model it would be interesting to understand if the heavier than Gaussian tails that we establish are in fact optimal. Furthermore, applying our current methods, it appears that the low integrability of the solution, as compared with stochastic reaction-diffusion equations, imposes a significant barrier to considering $m>2$ in \eqref{eq:1dRepEquation}. Understanding whether this is a genuine issue or simply one of methodology would be interesting. Going further, one could naturally ask if it is possible to extend the well-known global well-posedness of attractive chemotaxis models in one dimension, \cite{hillen_potapov_04,osaki_yagi_01}, to the stochastic case with white noise forcing. In addition, since chemotaxis is typically observed in two or three dimensions, extending the above study to higher dimensional versions would certainly be of interest. Using simple power counting one would expect \eqref{eq:1dRepEquation} to be sub-critical in the sense of regularity structures, \cite{hairer_14_RegStruct}, on $\mbR_+\times \mbT^d$ for $d<4$. While local analysis is tractable in the presence of white-noise forcing, establishing global well-posedness for the repulsive model in dimension two seems challenging by our current methods, see \cite[Ch.~7]{mayorcas_20}. Finally, beyond the additive noise case, both multiplicative and conservative noise models are highly relevant in the context of chemotaxis. Extending our results to those cases, especially in a two or three dimensional setting, is a natural goal.

In the remainder of the introduction we summarise some notation in Section \ref{subsec:Notation}, and present our main results in Section \ref{subsec:MainResults}. In Section \ref{sec:StatSHE} we recap the definitions and properties of the linear stochastic heat equation, which plays a key role in our analysis. In Section \ref{sec:1dLocWP} we obtain local well-posedness of \eqref{eq:1dRepEquation} and some ancillary properties of the solution. Section \ref{sec:1dAPriori} contains the main contribution of this paper, establishing an a priori bound that is independent of the initial data and constitutes a \textit{coming down from infinity property}. This a priori bound enables us to establish global well-posedness, existence of invariant measures and the tail bound, \eqref{eq:LawTail}, below. In Section \ref{sec:InvariantMeasures} we establish the existence of invariant measures to \eqref{eq:1dRepEquation} and closely following the approach of \cite{tsatsoulis_weber_18}, we establish the strong Feller property, irreducibility, and exponential ergodicity of the associated semi-group. The appendices, \ref{app:HolderBesovSpaces} and \ref{app:SHERegularity} respectively contain summaries and useful properties of the inhomogeneous H{\"o}lder--Besov spaces and the stochastic heat equation on $\mbT$.
\subsection{Notation}\label{subsec:Notation}
Let $\mbN=\{0,1,\ldots,\}$ denote the set of non-negative integers.
For $k\in\mbN$, we denote by $C^k(\mbT)$ the space of $k$ times continuously differentiable, $1$-periodic, real functions. We denote by $C^\infty(\mbT)$ the space of smooth, periodic functions with values in $\mbR$ and $\mcS'(\mbT)$ for its dual. For $k\in\mbN,\,p\in [1,\infty)$ (resp. $p=\infty$) we write $W^{k,p}(\mbT)$ for the spaces of periodic functions with $p$-integrable (resp. essentially bounded) weak derivatives up to $k^{\text{th}}$ order,
and write $L^p(\mbT)=W^{0,p}(\mbT)$.
We write $\mcB^\alpha_{p,q}(\mbT)$ for the Besov space associated to $\alpha\in \mbR$, $p,q\in [1,\infty]$
and use the shorthand $\mcC^\alpha(\mbT)=\mcB^\alpha_{\infty,\infty}(\mbT)$. Note that for $\alpha \in \mbN$ the spaces $\mcC^\alpha$ and $C^\alpha$ are not equal. See Appendix \ref{app:HolderBesovSpaces} for a full definition.

For $f \in \mcS'(\mbT)$ we denote its spatial mean by $\bar{f}:=\langle 1,f\rangle_{L^2(\mbT)}$. For $\mfm \in \mbR$ we write, for example, $\mcS'_\mfm(\mbT),C^k_\mfm(\mbT), \mcC^\alpha_\mfm(\mbT), L^p_{\mfm}(\mbT)$, for the corresponding spaces with the additional constraint that $\bar{f} =\mfm$. When the context is clear we drop dependence on the domain in order to lighten notation.

Given a Banach space $E$, a subset $I\subseteq [0,\infty)$ and $\kappa \in (0,1)$, we write $C_IE:= C(I;E)$ (resp. $\mcC^{\kappa}_I E := \mcC^{\kappa}(I;E)$) for the space of continuous (resp. $\kappa$-H{\"o}lder continuous) maps $f\colon I\rightarrow E$
equipped with the norm $\|f\|_{C_IE}:= \sup_{t\in I} \|f_t\|_E$ (resp. $\|f\|_{\mcC_I^{\kappa}E}:=\|f\|_{C_IE}+ \sup_{t
\neq s \in I} \frac{\|f_t-f_s\|_{E}}{|t-s|^{\kappa}}$). For $T>0$, we use the shorthand $C_TE=C_{[0,T]}E$ and $\mcC_T^\kappa E=\mcC_{[0,T]}^\kappa E$. Note that the norm $\|f\|_{\mcC_T^{\kappa}E}$
is equivalent to $\|f_0\|_E + \sup_{t\neq s \in [0,T]}\frac{\|f_t-f_s\|_E}{|t-s|^\kappa}$.
For $\eta>0$ we let $C_{\eta;T}E:= C_\eta((0,T];E)$ be the set of continuous functions $f:(0,T]\rightarrow E$
such that
\begin{equation*}
\|f\|_{C_{\eta;T}E} := \sup_{ t \in (0,T]} (t^\eta\wedge 1)\|f_t\|_E < \infty.
\end{equation*}
For a Banach space, $E$, equipped with its Borel $\sigma$-algebra, $\mcE$, we use the notation $\mcB_b(E)$, $\mcC_b(E)$ and $\mcC^1_b(E)$ respectively for the sets of bounded Borel measurable, continuous and continuously Fr{\'e}chet differentiable maps $\Phi:E\rightarrow \mbR$. We write $\mcP(E)$ for the set of probability measures on $E$ which we equip with the topology of weak convergence. For a sequence $(\mu_n)_{n\geq 1} \subset \mcP(E)$ we write $\mu_n \rightharpoonup \mu$ to indicate the weak convergence of $(\mu_n)_{n\geq 1}$ to $\mu\in \mcP(E)$. Using $\Pi_{\mu,\nu}$ to denote the set of all couplings between $\mu,\,\nu \in \mcP(E)$, we write the total variation distance as,
$$\|\mu-\nu\|_{\TV} := \sup_{A \in \mcE}|\mu(A)-\nu(A)| = \inf_{\pi \in 
\Pi_{\mu,\nu}} \iint_{E^2} \mathds{1}_{x=y}\dd \pi(x,y).$$
We write $\lesssim$ to indicate that an inequality holds up to a constant depending on quantities that we do not keep track of or are fixed throughout. 
When we do wish to emphasise the dependence on certain quantities, we either write $\lesssim_{K,v}$ or define $C:=C(\alpha,p,d)>0$ and write $\leq C$.
\subsection{Main Results}\label{subsec:MainResults}
We fix a filtered probability space $(\Omega,\mcF,(\mcF_t)_{t\geq 0},\mbP)$ carrying a mean-free, space-time white noise, $\xi$, defined in Section \ref{sec:StatSHE} below. We fix $\alpha_0\in \left(-\frac{1}{2},0\right)$, $\alpha \in \left(0,\alpha_0+\frac{1}{2}\right)$ and $\eta >0$ such that,
\begin{equation}\label{eq:exponentCriteria}
\frac{\alpha-\alpha_0}{2}<\eta< \frac{1}{4}.
\end{equation}
\begin{theorem}[Global Well-Posedness]\label{th:1dGWP}
Let $T>0$, $\mfm \in \mbR$, $\chi>0$ and $\zeta\in \mcC^{\alpha_0}_\mfm(\mbT)$. Then there exists a unique, probabilistically strong, mild solution, $u(\zeta)$, to \eqref{eq:1dRepEquation} such that $\mbP$-a.s. $u(\zeta)\in C_{\eta;T}\mcC_\mfm^{\alpha}(\mbT)$.
\end{theorem}
\begin{remark}
While we frame Theorem \ref{th:1dGWP} as a probabilistic statement, the proof is mainly deterministic and based on PDE techniques. In the course of the proof we additionally establish local Lipschitz continuity of the solution in both initial data and the noise, see Proposition \ref{prop:RemainderGWP}. This will be important in Section~\ref{subsec:FullSupport} when we establish full support of the law on $\mcC^{\alpha}_{\mfm}(\mbT)$.
\end{remark}
\begin{remark}
The statement of Theorem \ref{th:1dGWP} remains valid for $\eta =\frac{\alpha-\alpha_0}{2}$.
However, we take $\eta >\frac{\alpha-\alpha_0}{2}$ to simplify some statements and proofs below. In general we may think of $\eta$ as being arbitrarily close to $\frac{\alpha-\alpha_0}{2}$.
\end{remark}
Let $\mcL(u_t(\zeta)) := u_t(\zeta)\# \mbP \in \mcP\big(\mcC_{\mfm}^{\alpha}(\mbT)\big)$ denote the law of the solution, started from $\zeta \in \mcC^{\alpha_0}_{\mfm}(\mbT)$, at time $t>0$. 
\begin{theorem}[Exponential Ergodicity]\label{th:1dexpErgodic}
Let $\mfm \in \mbR$ and $\delta \in (0,1/2)$. Then there exists a unique measure $\nu \in \mcP\big(\mcC^{1/2-\delta}_\mfm(\mbT )\big)$ which is invariant for the semi-group associated to~\eqref{eq:1dRepEquation}.
Furthermore, $\nu$ has full support in $\mcC^{1/2-\delta}_\mfm$ and there exists $c>0$ such that for all $t>1$ and $\zeta \in \mcC^{\alpha_0}_\mfm$,
\begin{equation}\label{eq:LawConverge}
\|\mcL(u_t(\zeta)) -\nu\|_{\TV} \,\leq e^{-ct}.
\end{equation}
Finally, for any $p\in [1,\infty)$ there exists $\Lambda:= \Lambda(p,\delta)>0$ such that
\begin{equation}\label{eq:LawTail}
\int_{\mcC^{1/2-\delta}_{\mfm}} \exp\left(\Lambda \|f\|^{1-2\delta}_{L^p}\right)\,\nu (\dd f)<\infty.
\end{equation}
\end{theorem}
The proof of Theorem \ref{th:1dGWP} is completed at the end of Section \ref{sec:1dAPriori} and the proof of Theorem \ref{th:1dexpErgodic} is completed at the end of Section \ref{sec:InvariantMeasures}.
\section{Stochastic Heat Equation}\label{sec:StatSHE}
In what follows we will decompose the solution $u$ to \eqref{eq:1dRepEquation} into a regular and irregular part, with the irregular part being the solution to a linear stochastic heat equation (SHE). We briefly recap some necessary definitions of space-time white noise, the solution to the SHE and some of its properties.

\begin{definition}[White Noise]\label{def:WhiteNoiseDefinition}
Given an abstract probability space $(\Omega,\mcF,\mbP)$ we say that an $\mbR$-valued, stochastic process, indexed by $L^2(\mbR_+\times \mbT)$, $\{\xi(\varphi)\,:\,\varphi \in L^2(\mbR_+\times \mbT)\}$, defines a spatially mean-free space-time white noise if
\begin{enumerate}
\item \label{it:SpaceTimeMeanCovariance}$\mbE[\xi(\varphi)]=0$, $\mbE[\xi(\varphi)\xi(\varphi')] = \langle \varphi,\varphi'\rangle_{L^2(\mbR_+\times \mbT)}$ for all $\varphi,\,\varphi' \in L^2(\mbR_+\times \mbT)$,
\item \label{it:NoiseMeanFree}$ \xi(\psi\otimes 1)=0$ for all $\psi \in L^2(\mbR_+)$, $\mbP$-a.s.
\end{enumerate}
\end{definition}
For $\varphi\in L^2(\mbR_+ \times \mbT)$ we write $\xi(\varphi)$ in the convenient form of a stochastic integral,
\begin{equation}\label{eq:StochasticIntegrals}
\xi(\varphi) = \int_{\mbR_+} \int_{\mbT} \varphi(t,x) \xi(\dd t,\,\dd x),
\end{equation}
even though $\xi$ is almost surely not a measure. We refer to \cite[Ch. 4]{daprato_zabczyk_14} for further details
and a proper construction of \eqref{eq:StochasticIntegrals}.

We define the filtration
\begin{equation*}
(\tilde{\mcF}_t)_{t\geq 0} := \sigma \left(\left\{ \xi(\varphi)\,:\, \varphi \in L^2(\mbR_+ \times \mbT),\,\varphi |_{(t,\infty)} \equiv 0 \right\}\right),
\end{equation*}
and let $(\mcF_t)_{t\geq 0}$ denote its usual augmentation. 
\begin{definition}\label{def:SHE}
Let $0\leq t_0<T$ and $\mcH$ be the periodic heat kernel on $\mbT$ defined in \eqref{eq:PeriodicHeatKernel}. We say that the $\mcS'(\mbT)$ valued process,
\begin{equation}\label{eq:StochasticIntegral}
(t_0,T] \ni t\mapsto v_{t_0,t}(\phi):=\int_{t_0}^t \int_{\mbT}(\mcH_{t-r}\ast \phi)(x)\,\xi(\dd r,\dd x) \quad \,\forall \phi \in C^\infty(\mbT),
\end{equation}
is a mild solution to the SHE, with zero initial condition at $t=t_0$.
\end{definition}
\begin{theorem}\label{th:MarkovSHESpaceTimeRegular}
The process $[t_0,T]\ni t\mapsto v_{t_0,t}$ is, continuous, Markov, $(\mcF_t)_{t \in  [t_0,T]}$ adapted and for any $\alpha<1/2$, $\kappa \in [0,1/2)$ and $p\geq 1$ there exists $C:=C(T,\alpha,\kappa,p)>0$ such that, 
\begin{equation}\label{eq:MarkovSHESpaceTimeRegular}
\mbE\left[ \sup_{t\,\neq \,s \in [t_0,T]} \frac{\|v_{t_0,t}-v_{t_0,s}\|^p_{\mcC^{\alpha-2\kappa}}}{|t-s|^{p\kappa}} \right]<C.
\end{equation}
Furthermore, for any $0\leq t_0 < t\leq T$, the law of $v_{t_0,t}$ depends only on $|t-t_0|$.
\end{theorem}
\begin{proof}
See Appendix \ref{app:SHERegularity}.
\end{proof}
\section{Local Well-Posedness}\label{sec:1dLocWP}
Throughout this section, we fix $T>0$ and $\chi\in\mbR$ (not necessarily positive).
From Theorem \ref{th:MarkovSHESpaceTimeRegular} we see that $\mbP$-a.s. the map $t\mapsto v_{0,t}$ is finite only in $ C_T\mcC_0^{\alpha}$, for $\alpha<1/2$. We therefore cannot expect to find solutions to \eqref{eq:1dRepEquation} of any higher regularity. Obtaining the a priori estimates in Section \ref{sec:1dAPriori} requires at least one degree of spatial regularity. Therefore, although the equation is not singular itself, we solve for a remainder process with higher regularity using a Da Prato--Debussche trick,~\cite{daPrato_debussche_03_quantization}. Concretely we decompose the solution as $u_t:=w_t+Z_t$, where $t\mapsto Z_t$ is a deterministic function taking values in the H{\"o}lder--Besov space $\mcC_0^{\alpha}$ with zero spatial mean. Subsequently we will take $Z_t$ to be a $\mbP$-a.s. realisation of the stochastic heat equation $t\mapsto v_{0,t}$. The unknown, $w$, solves,
\begin{equation}\label{eq:1dRepRemainder}
\begin{cases}
\partial_t w- \partial_{xx} w = \chi \partial_x((w+Z)^2\partial_x \rho_{w+Z}), &\text { on } \mbR_+ \times \mbT,\\
-\partial_{xx} \rho_{w+Z} = w-\bar{w}+Z, &\text{ on }\mbT,\\
w\tzero = \zeta, &\text{ on }\mbT.
\end{cases}
\end{equation}
For the rest of this section, we fix $\zeta\in \mcC^{\alpha_0}(\mbT)$ and $Z\in C_T\mcC^{\alpha}_0$.

We first show that under suitable regularity assumptions on $w$, the right hand side of the first equation in \eqref{eq:1dRepRemainder} is a well-defined element of $C_{\eta;T}\mcC^\alpha(\mbT)$.
\begin{lemma}\label{lem:PsiWellDefined}
Let $w\in C_{\eta;T}\mcC^{\alpha}$.
Then the map $w\mapsto \Psi w$, defined for any $t\in (0,T]$, by
\begin{equation}
    (\Psi w)_t := e^{t\Delta} \zeta + \int_0^t e^{(t-s)\Delta} \chi
    \partial_x ((w_s+Z_s)^2 \partial_x \rho_{w_s+Z_s})\dd s,
\end{equation}
is well-defined from $C_{\eta;T}\mcC^{\alpha}$ to itself.
\end{lemma}
\begin{proof}
Applying \eqref{eq:HeatFlow}, for any $t >0$, we have that
\begin{equation}\label{eq:InitialCondHeatKernel}
\|e^{t\Delta}\zeta\|_{\mcC^{\alpha}} \lesssim t^{-\frac{\alpha-\alpha_0}{2}}\|\zeta\|_{\mcC^{\alpha_0}}.
\end{equation}
Concerning the integral term, expanding the square and applying Theorem~\ref{th:BesovProduct}, along with \eqref{eq:BesovElliptic} we obtain the bounds
\begin{equation*}
\begin{aligned}
\|w^2_s\partial_x \rho_{w_s+Z_s}\|_{\mcC^{\alpha}}&\leq \|w\|^2_{\mcC^{\alpha}}\|\partial_x \rho_{w_s+Z_s}\|_{\mcC^{\alpha}}\lesssim \|w_s\|^3_{\mcC^{\alpha}}+\|w_s\|^2_{\mcC^{\alpha}}\|Z_s\|_{\mcC^{\alpha}}\\
2\|w_sZ_s\partial_x \rho_{w_s+Z_s}\|_{\mcC^{\alpha}} &\leq 2\|w\|_{\mcC^{\alpha}}\|Z_s\|_{\mcC^{\alpha}}\| \partial_x \rho_{w_s+Z_s}\|_{\mcC^{\alpha}} \lesssim 2\|w_s\|^2_{\mcC^\alpha}\|Z_s\|_{\mcC^\alpha} + 2\|w_s\|_{\mcC^{\alpha}}\|Z_s\|^2_{\mcC^{\alpha}}\\
\|Z^2_s\partial_x \rho_{w_s+Z_s}\|_{\mcC^{\alpha}}&\leq 	\|Z_s\|^2_{\mcC^{\alpha}}\|\partial_x \rho_{w_s+Z_s}\|_{\mcC^{\alpha}}\lesssim \|w_s\|_{\mcC^{\alpha}}\|Z_s\|^2_{\mcC^\alpha}+\|Z_s\|^3_{\mcC^{\alpha}}.
\end{aligned}
\end{equation*}
Combining these yields that
\begin{equation}\label{eq:1dNonLinearBounds}
\|(w_s+Z_s)^2\partial_x \rho_{w_s+Z_s}\|_{\mcC^{\alpha}}\lesssim \sum_{k=0}^{3}\binom{3}{k} \|w_s\|^{3-k}_{\mcC^{\alpha}}\|Z_s\|^k_{\mcC^{\alpha}}.
\end{equation}
Therefore, applying \eqref{eq:HeatFlow}, \eqref{eq:BesovDerivative} for any $s<t\in (0,T\wedge 1]$ we have
\begin{equation*}
\begin{aligned}
\|e^{(t-s)\Delta}\partial_x ((w_s+Z_s)^2\partial_x \rho_{w_s+Z_s})\|_{\mcC^{\alpha}}&\lesssim (t-s)^{-\frac{1}{2}}\| (w_s+Z_s)^2\partial_x \rho_{w_s+Z_s}\|_{\mcC^{\alpha}}\\
&\lesssim (t-s)^{-\frac{1}{2}}s^{-3\eta}\max_{k=0,\ldots,3}\left\{\|w\|^{3-k}_{C_{\eta;T}\mcC^{\alpha}}\|Z\|^{k}_{C_T\mcC^{\alpha}} \right\}.
\end{aligned}
\end{equation*}
Since $\eta <\frac{1}{4}$ we may integrate $s^{-3\eta}$ near $0$ and so for $t \in (0,T\wedge 1]$ we have
\begin{equation}\label{eq:PsiGenBnd}
t^\eta\|(\Psi w)_t\|_{\mcC^{\alpha}}\lesssim t^{\eta-\frac{\alpha-\alpha_0}{2}}\|\zeta\|_{\mcC^{\alpha_0}}+t^{\frac{1}{2}-2\eta}\max_{k=0,\ldots,3}\left\{\|w\|^{3-k}_{C_{\eta;T}\mcC^{\alpha}}\|Z\|^{k}_{C_T\mcC^{\alpha}} \right\}.
\end{equation}
where both exponents are positive due to \eqref{eq:exponentCriteria} and the norms on the right hand side are finite by assumption. For $t>1$ one may argue in almost exactly the same way, only splitting the time integral at $t=1$ and replacing the multiplication by $t^\eta$ with $(t\wedge 1)^\eta$. 
\end{proof}
\begin{definition}[Mild Solutions to \eqref{eq:1dRepRemainder}]\label{def:1dRemainderMildSols}
We say that $w\in C_{\eta;T}\mcC^\alpha$ is a mild solution to \eqref{eq:1dRepRemainder} on $[0,T]$ (started from $\zeta$ and driven by $Z$)
if for every $t\in (0,T]$,
\begin{equation}\label{eq:1dRemMildSol}
w_t = e^{t\Delta}\zeta+\int_0^t e^{(t-s)\Delta}\chi\partial_x  \left( (w_s+Z_s)^2\partial_x \rho_{w_s+Z_s} \right)\,\dd s.
\end{equation}
\end{definition}

\begin{remark}
Lemma \ref{lem:PsiWellDefined} demonstrates that for any solution, the right hand side of \eqref{eq:1dRemMildSol} is well-defined.
\end{remark}
\begin{theorem}[Local Well-Posedness of \eqref{eq:1dRepRemainder}]\label{th:1dRemainderLWP}
Let $\mfR\geq1$ be such that
$\|Z\|_{C_T\mcC^{\alpha}}^3 + \|\zeta\|_{\mcC^{\alpha_0}(\mbT)}<\mfR$.
Then there exists $C>0$, independent of $\mfR$, $\zeta$, and $Z$,
such that \eqref{eq:1dRepRemainder} has a unique mild solution $w \in C_{\eta;T_*}\mcC^\alpha$
where
\begin{equation}\label{eq:1dT(R)Def}
T_* = \left(\frac{1}{C\mfR}\right)^{\frac{1}{\theta}}\wedge T, \,\, \text{ with }\,\,\theta:=\left(\eta -\frac{\alpha-\alpha_0}{2}\right)\wedge \left(\frac{1}{2}-2\eta\right).
\end{equation}
Furthermore
\begin{equation}\label{eq:bound_by_1}
\sup_{t \in (0,T_*]} t^\eta \|w_t\|_{\mcC^{\alpha}} \leq 1
\end{equation}
and
\begin{equation}\label{eq:conv_to_zeta}
\lim_{t\to0} \|w_t-\zeta\|_{\mcC^{\alpha_0}}=0.
\end{equation}
\end{theorem}
\begin{proof}
Denoting
\begin{equation*}
B_{T_*} := \Big\{ w \in C((0,T_*];\mcC^{\alpha}(\mbT))\,:\,
\sup_{t \in (0,T_*]}t^\eta\|w_t\|_{\mcC^{\alpha}}\leq  1\Big\},
\end{equation*}
we will show that 
\begin{equation*}
w \mapsto (\Psi w)_t:=e^{t\Delta}\zeta+\int_0^t e^{(t-s)\Delta}\chi\partial_x \left( (w_s+Z_s)^2\partial_x \rho_{w_s+Z_s} \right)\,\dd s,
\end{equation*}
is a contraction on $B_{T_*}$ for $T_*$ defined by~\eqref{eq:1dT(R)Def} for $C>0$ sufficiently large.
By \eqref{eq:PsiGenBnd}, for $w\in B_{T_*}$ and $t\in (0,T_*]$, there exists $C>0$ such that
\begin{align*}
t^\eta\|(\Psi w)_t\|_{\mcC^{\alpha}} &\leq C t^{\left(\eta -\frac{\alpha-\alpha_0}{2}\right)\wedge \left(\frac{1}{2}-2\eta\right)}\mfR,
\end{align*}
and so $\Psi$ maps $B_{T_*}$ into itself for $T_*$ defined by~\eqref{eq:1dT(R)Def}.
To show that $\Psi$ is a contraction we let $w,\,\tilde{w} \in B_{T_*}$. For any $s\in (0,T_*]$, using similar steps to those in the proof of \eqref{eq:PsiGenBnd}, we have that
\begin{align*}
\|(w_s+Z_s)^2\partial_x \rho_{w_s+Z_s}-
(\tilde{w}_s+Z_s)^2&\partial_x \rho_{\tilde{w}_s+Z_s}\|_{\mcC^{\alpha}}\\ & \lesssim \|w_s-\tilde{w}_s\|_{\mcC^{\alpha}}\sum_{k=0}^2 \binom{2}{k} \left(\|w_s\|_{\mcC^{\alpha}}\vee\|\tilde{w}_s\|_{\mcC^{\alpha}}\right)^{2-k}\|Z_s\|^{k}_{\mcC^{\alpha}}\\
&\lesssim
\mfR s^{-2\eta}\|w_s-\tilde{w}_s\|_{\mcC^{\alpha}}.
\end{align*}
So then, for any $t \in (0,T_*]$, we have
\begin{align*}
t^\eta\|(\Psi w)_t-(\Psi \tilde{w})_t\|_{\mcC^{\alpha}}
&\leq t^\eta\int_0^t \|e^{(t-s)\Delta} \chi\partial_x
\big(
(w_s+Z_s)^2\partial_x \rho_{w_s+Z_s}\\
&\qquad\qquad\qquad\qquad\qquad\qquad -(\tilde{w}_s+Z_s)^2\partial_x \rho_{\tilde{w}_s+Z_s}
\big)
\|_{\mcC^{\alpha}}\,\dd s
\\
&\lesssim t^\eta\int_0^t (t-s)^{-\frac{1}{2}}\|(w_s+Z_s)^2\partial_x \rho_{w_s+Z_s}-(\tilde{w}_s+Z_s)^2\partial_x \rho_{\tilde{w}_s+Z_s}\|_{\mcC^{\alpha}}\,\dd s
\\
&\lesssim \mfR t^{\frac{1}{2}-2\eta}\|w-\tilde{w}\|_{C_{\eta;t}\mcC^{\alpha}}.
\end{align*}
It follows that there exists $C>0$ such that, for $T_*$ given by~\eqref{eq:1dT(R)Def},
\begin{align*}
\|\Psi w -\Psi \tilde{w}\|_{C_{\eta;T_*}\mcC^\alpha}\, &\leq C\mfR (T_*)^{\frac{1}{2}-2\eta}\| w - \tilde{w}\|_{C_{\eta;T_*}\mcC^\alpha}\\
&\leq \frac12 \| w - \tilde{w}\|_{C_{\eta;T_*}\mcC^\alpha}.
\end{align*}
Hence $\Psi$ is a contraction on $B_{T_*}$,
and therefore there exists a unique fixed point $w \in B_{T_*}$ of $\Psi$ which, by construction, is a mild solution to \eqref{eq:1dRepRemainder} in the sense of Definition \ref{def:1dRemainderMildSols}
and satisfies~\eqref{eq:bound_by_1}.

To show that $w$ is the unique solution in all of $C_{\eta;T_*}\mcC^{\alpha}$, let $\tilde w$ be another mild solutions of \eqref{eq:1dRepRemainder}.
Then, by a similar argument to the above, there exists a $\tilde{T}(\|w\|_{C_{\eta;T_*}\mcC^\alpha},\|\tilde{w}\|_{C_{\eta;T_*}\mcC^\alpha})=:\tilde{T}\in (0,T_*]$ such that $w,\tilde{w}\in B_{\tilde{T}}$. Since both must be fixed points of $\Psi$ on $B_{\tilde{T}}$ we have that $w=\tilde{w}$ on $[0,\tilde{T}]$. Iterating the argument, using the same $\tilde{T}$ at each step, shows that $w=\tilde{w}$ on $[0,T_*]$.

To show~\eqref{eq:conv_to_zeta}, observe first that $\lim_{t\to0}\|e^{t\Delta}\zeta-\zeta\|_{\mcC^{\alpha_0}}=0$ by Remark~\ref{rem:HeatSemiGroupContAtZero} and so it only remains to show that the integral term converges to zero in $\mcC^{\alpha_0}$.
Consider $\tilde{\eta} \in \left(\frac{\alpha-\alpha_0}{2},\eta\right)$.
Since $C_{\tilde{\eta};T_*}\mcC^{\alpha} \hookrightarrow C_{\eta;T_*}\mcC^{\alpha}$,
applying what we have proved so far to $(\alpha_0,\alpha,\tilde\eta)$ in place of $(\alpha_0,\alpha,\eta)$,
we see that, for $S>0$ sufficiently small,
$w$ is also the unique mild solution to~\eqref{eq:1dRepRemainder} in $C_{\tilde\eta;S}\mcC^\alpha$,
and that
$\sup_{t\in(0,S]}t^{\tilde\eta}\|w_t\|_{\mcC^\alpha} \leq 1$.
Applying \eqref{eq:HeatFlow} and \eqref{eq:1dNonLinearBounds}, for all $t\in (0,S]$
\begin{align*}
    \int_0^t \|e^{(t-s)\Delta} \partial_x((w_s+Z_s)^2\partial_x\rho_{w_s+Z_s}\|_{\mcC^{\alpha_0}}\dd s &\lesssim \int_0^t (t-s)^{-\frac{\alpha_0 -\alpha+1}{2}} s^{-3\tilde{\eta}} \dd s\\
    &= t^{\frac{1}{2}+\frac{\alpha-\alpha_0}{2}-3\tilde{\eta}},
\end{align*}
where we used the fact that $\frac{\alpha_0-\alpha+1}{2}\vee 3\tilde{\eta}<1$ to evaluate the integral. We now choose $\tilde{\eta}$ sufficiently close to $\frac{\alpha-\alpha_0}{2}$ so that $\frac{1}{2}+\frac{\alpha-\alpha_0}{2}-3\tilde{\eta}>0$, from which~\eqref{eq:conv_to_zeta}
follows.
\end{proof}
\begin{lemma}\label{lem:spatial_reg}
Suppose that $w\in C_{\eta;T}\mcC^\alpha$ is a mild solution to~\eqref{eq:1dRepRemainder}.
Then for all $t_0\in(0,T)$, $\beta\in (\alpha,\alpha+1)$ and $\kappa \in (0,1)$,
\begin{equation}\label{eq:1dRemainderTimeHolder}
\sup_{t\,\neq\, s\, \in [t_0,T]}\frac{\|w_t-w_s
\|_{\mcC^{\beta-2\kappa}}}{|t-s|^{\kappa}}<\infty.
\end{equation}
\end{lemma}
\begin{proof}
Applying \eqref{eq:HeatFlow}
gives
\begin{equation}\label{eq:1dRemainderHigherRegInitialCond}
\|e^{t\Delta}\zeta\|_{\mcC^\beta}\lesssim t^{-\frac{\beta-\alpha_0}{2}}\|\zeta\|_{\mcC^{\alpha_0}}.
\end{equation}
Regarding the integral term, applying similar steps to those in the proof of \eqref{eq:PsiGenBnd} and the assumption $\frac{1+\beta-\alpha}{2}<1$,
\begin{align}
\int_{0}^t\|e^{(t-s)\Delta}\partial_x \left( (w_s+Z_{s})^2\partial_x \rho_{w_s+Z_s}\right)\|_{\mcC^{\beta}}\,\dd s
&\lesssim \int_{0}^t (t-s)^{-\frac{1+\beta-\alpha}{2}}s^{-3\eta}\,\dd s \notag \\
&= t^{\frac{1}{2}-\frac{\beta-\alpha}{2}-3\eta} \label{eq:1dRemainderHigherRegIntegral}.
\end{align}
It follows that
\begin{equation*}
\sup_{t \in [t_0,T_*]} \|w_t\|_{\mcC^{\beta}} <\infty.
\end{equation*}
Now, let $0<t_0<s<t\leq T$ and $\kappa \in [0,1)$. Applying the triangle inequality and using the semi-group property, we have that
\begin{equation*}
\begin{aligned}
\|w_t-w_s\|_{\mcC^{\beta-2\kappa}} &\leq \left\|(e^{(t-s)-1)\Delta}w_{s}\right\|_{\mcC^{\beta-2\kappa}} \\
&\quad + \int_{s}^{t} \left\|e^{(t-r)\Delta}\partial_x ((w_{r}+Z_{r})^2\partial_x \rho_{w_{r}+Z_{r}})\right\|_{\mcC^{\beta-2\kappa}} \,\dd r\\
&=: I_1(s,t)+\int_s^{t} I_2(r,t)\,\dd r.
\end{aligned}
\end{equation*} 
Using~\eqref{eq:HeatFlow} and~\eqref{eq:HeatFlowMinusId},
\begin{align}
I_1(s,t) \lesssim\left\|\left(e^{(t-s)\Delta} - 1\right)w_{s}\right\|_{\mcC^{\beta-2\kappa}}   \lesssim (t-s)^{\kappa}\|w_{s}\|_{\mcC^\beta}\label{eq:1dRemainderTimeHolderInitial}.
\end{align}
Since $\beta-1<\alpha$ and $\sup_{t \in [t_0,T_*]}\|Z_t\|_{\mcC^\alpha}+\|w_t\|_{\mcC^{\alpha}}<\infty$,
applying \eqref{eq:HeatFlow} followed by \eqref{eq:HeatFlowMinusId} and \eqref{eq:1dNonLinearBounds},
we obtain
\begin{align*}
I_2(r,t) \lesssim (t-r)^{-\frac{1+\beta-2\kappa-(\beta-1)}{2}}\left\| ((w_{r}+Z_{r})^2\partial_x \rho_{w_{r}+Z_{r}})\right\|_{\mcC^{\beta-1}}\lesssim (t-r)^{-1+\kappa}.
\end{align*}
Therefore
\begin{equation}\label{eq:1dRemainderTimeHolderIntegral1}
\int_s^{t} I_2(r,t)\,\dd r \lesssim (t-s)^{\kappa}.
\end{equation}
Combining~\eqref{eq:1dRemainderTimeHolderInitial} and~\eqref{eq:1dRemainderTimeHolderIntegral1} we obtain~\eqref{eq:1dRemainderTimeHolder}.
\end{proof}
\begin{lemma}[Mild Solutions are Weak Solutions] \label{lem:MildSolsAreWeakSols}
Let $t_0 \in [0,T)$ and $w\in C_{\eta;T}\mcC^{\alpha}$ be a mild solution to \eqref{eq:1dRepRemainder} on $[0,T]$. Then for any $\phi \in C^\infty(\mbT)$ and $t \in [t_0,T]$,
\begin{equation}\label{eq:weak_sols}
\langle w_t,\phi\rangle - \langle w_{t_0},\phi\rangle = -\int_{t_0}^t \langle \partial_x w_s
+\chi(w_s+Z_s)^2\partial_x \rho_{w_s+Z_s},\partial_x \phi\rangle\,\dd s.
\end{equation}
Moreover,~\eqref{eq:weak_sols} holds for all $t_0\in(0,T)$ and $\phi\in C^1(\mbT)$.
\end{lemma}
\begin{proof}
From the mild form of the equation started from data $w_{t_0}$ at $t=t_0$, it follows that for $\phi \in C^{\infty}(\mbT)$
\begin{align*}
\int_{t_0}^t \langle w_s,\partial_{xx} \phi\rangle \,\dd s
&= \int_{t_0}^t \left \langle e^{s\Delta}w_{t_0},\partial_{xx} \phi\right\rangle\,\dd s
\\
&\qquad\qquad
+ \int_{t_0}^t \int_{t_0}^s\left\langle e^{(s-r)\Delta} \chi\partial_x ((w_r+Z_r)^2\partial_x \rho_{w_r+Z_r}),\partial_{xx} \phi \right\rangle \,\dd r\,\dd s.
\end{align*}
Integrating by parts in the first term on the right hand side gives 
\begin{equation}\label{eq:InitialConditionWeakForm}
\int_{t_0}^t \left \langle e^{s\Delta}w_{t_0},\partial_{xx} \phi\right\rangle\,\dd s = \int_{t_0}^t \langle \partial_{xx} e^{s\Delta}w_{t_0},\phi\rangle\,\dd s  = \langle e^{t\Delta}w_{t_0},\phi\rangle -\langle w_{t_0},\phi\rangle .
\end{equation}
Using the same argument for the second term on the right hand side we have that
\begin{equation*}
\left\langle e^{(s-r)\Delta} \partial_x ((w_r+Z_r)^2\partial_x \rho_{w_r+Z_r}),\partial_{xx} \phi \right\rangle = \partial_s \left\langle e^{(s-r)\Delta} \partial_x ((w_r+Z_r)^2\partial_x \rho_{w_r+Z_r}), \phi \right\rangle.
\end{equation*}
Changing the order of integration gives
\begin{align*}
\int_{t_0}^t\int_{t_0}^s\langle e^{(s-r)\Delta} \partial_x ((w_r+Z_r)^2&\partial_x \rho_{w_r+Z_r}),\partial_{xx} \phi \rangle\,\dd r\,\dd s\\
&=\int_{t_0}^t \langle e^{(t-r)\Delta} \partial_x ((w_t+Z_t)^2\partial_x \rho_{w_t+Z_t}),\phi\rangle\,\dd r\\
&\quad \quad- \int_{t_0}^t \langle \partial_x ((w_r+Z_r)^2\partial_x \rho_{w_r+Z_r}),\phi\rangle\,\dd r.
\end{align*}
Putting this together with \eqref{eq:InitialConditionWeakForm} gives that
\begin{align*}
\int_{t_0}^t \langle w_s,\partial_{xx} \phi\rangle \,\dd s &= \langle e^{t\Delta}w_{t_0},\phi\rangle -\langle w_{t_0},\phi\rangle 
+
\int_{t_0}^t \langle e^{(t-r)\Delta} 
\chi\partial_x ((w_r+Z_r)^2\partial_x \rho_{w_r+Z_r}),\phi\rangle\,\dd r\\
&\quad -
\int_{t_0}^t \langle
\chi\partial_x ((w_r+Z_r)^2\partial_x \rho_{w_r+Z_r}),\phi\rangle\,\dd r.
\end{align*}
Rearranging and integrating the left hand side by parts once proves~\eqref{eq:weak_sols} for any $\phi \in C^\infty(\mbT)$.
For $t_0\in(0,T)$,
by Theorem \ref{th:1dRemainderLWP},
we have that $\sup_{t\in [t_0,T]}\|w_t\|_{C^1}<\infty$, so that the map $C^1(\mbT)\ni \phi \mapsto \int_{t_0}^t \langle \partial_x w_s + (w_s+Z_s)^2\partial_x\rho_{w_s+Z_s},\partial_x \phi\rangle \dd s$ is continuous.
Approximating $\phi \in C^1(\mbT)$ by a smooth sequence gives the result.
\end{proof}
It follows from the proof of Lemma \ref{lem:MildSolsAreWeakSols} that any solution to \eqref{eq:1dRepRemainder} has constant spatial mean.
\begin{corollary}\label{cor:1dSolConstantMean}
Let $t\in [0,T]$ and $w \in C_{\eta;T}\mcC^{\alpha}$ be a mild solution to \eqref{eq:1dRepRemainder}.
Then $\bar{w}_t = \bar{\zeta}$.
\end{corollary}
\begin{proof}
Applying Lemma \ref{lem:MildSolsAreWeakSols} with $\phi\equiv 1$ and $t_0=0$ gives $\bar{w}_t=\langle w_t,1\rangle = \langle \zeta,1\rangle =\bar{\zeta}$.
\end{proof}
\section{A Priori Estimate and Global Well-Posedness}\label{sec:1dAPriori}
In this section we make use of the specific sign choice $\chi >0$ in \eqref{eq:1dRepEquation} to obtain an a priori estimate on the solution in Theorem \ref{th:1dRemainderAPriori}.
Throughout this section we fix $T>0$, and a zero spatial mean function $Z \in C_T\mcC^{\alpha}_0$.
\begin{proposition}[Testing the Equation]\label{prop:1dRemainderTest} Let $0<t_0<t\leq T$, $p\geq 2$ an integer, $\zeta\in \mcC^{\alpha_0}(\mbT)$,  and $w(\zeta) \in C_{\eta;T}\mcC^{\alpha}$ be a mild solution to \eqref{eq:1dRepRemainder} on $[0,T]$.
Then
\begin{equation}\label{eq:1dRemainderTestEq}
\begin{aligned}
\frac{1}{p(p-1)}\left(\|w_t\|_{L^p}^p -\|w_{t_0}\|_{L^p}^p\right) &= -\int_{t_0}^t \langle w_s^{p-2}\partial_x w_s,\partial_x w_s\rangle\,\dd s\\
&\quad-\int_{t_0}^t\langle \chi(w_s+Z_s)^2\partial_x \rho_{w_s+Z_s},w_s^{p-2}\partial_x w_s\rangle \,\dd s.
\end{aligned}
\end{equation}
\end{proposition}
\begin{proof}
Let $\beta \in (1,\alpha+1)$ and $\kappa \in \left(1/2,\beta/2\right)$.
By Lemma~\ref{lem:spatial_reg},
\begin{equation}\label{eq:SolTimeHolder}
\sup_{s\,\neq \,r \in [t_0,T]} \frac{\|w_s-w_r\|_{L^\infty}}{|s-r|^\kappa}\leq \sup_{s\,\neq \,r \in [t_0,T]} \frac{\|w_s-w_r\|_{\mcC^{\beta-2\kappa}}}{|s-r|^\kappa} < \infty.
\end{equation}
Now, let $\varphi \in C^{\infty}(\mbR)$ and observe that for any $t_0\leq r<s\leq T$ we have the identity 
\begin{align*}
\langle w_s,\varphi(w_s)\rangle - \langle w_r,\varphi(w_r)\rangle& = \langle w_s,\varphi(w_r)\rangle - \langle w_r,\varphi(w_r)\rangle + \langle w_s,\varphi(w_s)-\varphi(w_r)\rangle.
\end{align*}
So for any $t\in (t_0,T]$, and $n\geq 2$, defining a family of partitions, by setting $t_i = t_0+i(t-t_0)/n$, for $i=0,\ldots,n$ and applying Lemma \ref{lem:MildSolsAreWeakSols} with the test function $\varphi(w_{t_i})$, we have 
\begin{align}
\langle w_t,\varphi(w_t)\rangle - \langle w_{t_0},\varphi(w_{t_0})\rangle&=-\sum_{i=0}^{n-1}\int_{t_i}^{t_{i+1}}
\langle \partial_x w_s
+\chi(w_s+Z_s)^2\partial_x \rho_{w_s+Z_s},\partial_x \varphi(w_{t_i})\rangle \dd s \notag \\
&\quad +\sum_{i=0}^{n-1} \left\langle w_{t_{i+1}},\varphi(w_{t_{i+1}})-\varphi(w_{t_i})\right\rangle\notag \\
&=:-I_n(t)+R_n(t) \notag.
\end{align}
By continuity of $s\mapsto \partial_x\varphi(w_s)$ we see that
\begin{equation}\label{eq:ILimit}
    \lim_{n\rightarrow \infty} I_n(t)= \int_0^t \langle \partial_x w_s
    + \chi(w_s+Z_s)^2\partial_x \rho_{w_s+Z_s},\partial_x \varphi(w_{s})\rangle \dd s.
\end{equation}
For $R^n(t)$, Taylor's formula gives
\begin{align*}
R_n(t) &= \sum_{i=0}^{n-1}\left[ \langle w_{t_{i+1}},(w_{t_{t+i}}-w_{t_i})\varphi'( w_{t_i})\rangle + \Big\langle w_{t_{i+1}},\int_{w_{t_i}}^{w_{t_{i+1}}} \varphi''(y)(w_{t_{i+1}}-y)\,\dd y\Big\rangle\right]
\\
&=:R_{n;1}(t)+R_{n;2}(t).
\end{align*}
We show that $R^n_2(t)$ converges to zero. Using \eqref{eq:SolTimeHolder}, we have the bound
\begin{align*}
|R_{n;2}(t)|&\lesssim
\sup_{s\in[t_0,t]}\|w_s\|_{L^\infty}\sup_{\substack{i = 0,\ldots,n-1\\y \in [w_{t_{i}},w_{t_{i+1}}]}}|\varphi''(y)|\,\sum_{i=0}^{n-1} \|w_{t_{i+1}}-w_{t_i}\|^2_{L^\infty} \overset{\tiny \eqref{eq:SolTimeHolder}}{\lesssim} \sum_{i=0}^{n-1} |t_{i+1}-t_i|^{2\kappa},
\end{align*}
where we used that $\varphi''$ is continuous. Regarding $R_{n;1}(t)$, we may apply Lemma \ref{lem:MildSolsAreWeakSols} once again to each bracket, this time with the test function $w_{t_{i+1}}\varphi'(w_{t_i})$, to give that
\begin{align*}
R_{n;1}(t)	&= -\sum_{i=0}^{n-1} \int_{t_i}^{t_{i+1}} \langle \partial_x w_s
+\chi(w_s+Z_s)^2\partial_x \rho_{w_s+Z_s},\partial_x (w_{t_{i+1}}\varphi'(w_{t_i}))\rangle \,\dd s.
\end{align*}
So again, by regularity of the map $s \mapsto \partial_x w_s$ and $\varphi$, we have that
\begin{equation}\label{eq:RLimit}
\lim_{n\rightarrow \infty}R_n(t)= -\int_{t_0}^{t} \langle \partial_x w_s
+\chi(w_s+Z_s)^2\partial_x \rho_{w_s+Z_s},\partial_x (w_{s}\varphi'(w_{s}))\rangle \,\dd s.
\end{equation}
So combining \eqref{eq:ILimit} and \eqref{eq:RLimit} gives,
\begin{align*}
\langle w_t,\varphi(w_t)\rangle - \langle w_{t_0},\varphi(w_{t_0})\rangle
=
-\int_{t_0}^t \langle \partial_x w_r
&+
\chi(w_r+Z_r)^2\partial_x \rho_{w_r+Z_r},
\\
&\qquad
\left(2\varphi'(w_r)+ w_r \varphi''(w_r)\right)\partial_x w_r \rangle \,\dd r.
\end{align*}
For $\varphi(x):= x^{p-1}$ we have $2\varphi'(x) + x\varphi''(x)= p(p-1)x^{p-2}$ and
\begin{equation*}
\langle w_t,w_t^{p-1}\rangle - \langle w_{t_0},w_{t_0}^{p-1}\rangle = \|w_t\|_{L^p}^p-\|w_{t_0}\|_{L^p}^p,
\end{equation*}
which gives \eqref{eq:1dRemainderTestEq}.
\end{proof}
Observe that one of the terms appearing in the second term on the right hand side of \eqref{eq:1dRemainderTestEq} is
\begin{equation*}
-\langle \partial_x \rho_{w},w^p\partial_x w\rangle = -\frac{1}{p+1}\langle \partial_x \rho_{w},\partial_x w^{p+1}\rangle = -\frac{1}{p+1}\|w^{p+2}\|_{L^1} + \frac{1}{p+2}\langle \bar{w},w^{p+1} \rangle .
\end{equation*}
We will use this term in combination with the first term, $ -\|w_t^{p-2}|\partial_x w_t|^2\|_{L^1}$, coming from the Laplacian,
to obtain an a priori estimate in Theorem~\ref{th:1dRemainderAPriori} on $\|w_{t}\|^p_{L^p}$ that is independent of the initial data. To do so we make use of an ODE comparison lemma which can be found with proof as \cite[Lem.~3.8]{tsatsoulis_weber_18}.
\begin{lemma}[ODE Comparison]\label{lem:ODEComparison}
Let $\lambda>1$, $c_1,\,c_2>0$ and $f\colon[0,T]\rightarrow [0,\infty)$ be differentiable and satisfying
\begin{equation*}
f'(t) + c_1 f(t)^\lambda \leq c_2,
\end{equation*}
for every $t\in [0,T]$. Then for all $t\in(0,T]$
\begin{equation*}
f(t)\leq \max\Big\{
\Big(\frac{tc_1(\lambda-1)}{2}\Big)^{-\frac{1}{\lambda-1}},\,\left(\frac{c_2}{c_1}\right)^{\frac{1}{\lambda}}\Big\} .
\end{equation*}
\end{lemma}
\begin{remark}\label{rem:chi}
For a given initial data $\zeta \in \mcC^{\alpha_0}(\mbT)$, a mild solution $w$ on $[0,T]$ to~\eqref{eq:1dRepRemainder}, and a constant $c\in\mbR$ with $c\neq 0$, we have that $\tilde{w}_t=cw_t$ satisfies
\begin{equation*}
\tilde{w}_t = e^{t\Delta}\tilde\zeta
+ \int_0^t e^{(t-s)\Delta} c^{-2}\chi\partial_x
\big(
(\tilde w_s + \tilde Z_s)^2\partial_x\rho_{\tilde w_s+\tilde Z_s}
\big)\dd s\;,
\end{equation*}
where $\tilde \zeta = c\zeta$ and $\tilde Z_s = cZ_s$.
In particular, if $\chi>0$, then setting $c:=\chi^{1/2}$ we see that any solution to~\eqref{eq:1dRepRemainder} is equal, up to a scalar multiple, to a solution to~\eqref{eq:1dRepRemainder} with $(\zeta,Z,\chi)$ replaced by $(\tilde\zeta,\tilde Z,1)$.
Conversely, 
if $\bar\zeta\neq 0$, then setting $c:=1/\bar\zeta$
we see that any solution to~\eqref{eq:1dRepRemainder} is equal, up to a scalar multiple, to a solution to~\eqref{eq:1dRepRemainder} with $(\zeta,Z,\chi)$
replaced by $(\tilde\zeta,Z,c^2 \chi)$
and where now $\tilde \zeta\in\mcC^{\alpha_0}_1(\mbT)$.
\end{remark}

In light of Remark~\ref{rem:chi}, we phrase the following a priori estimate only for the case $\bar\zeta\in\{0,1\}$.
\begin{theorem}[A Priori Bound on the Remainder]\label{th:1dRemainderAPriori}
Let $p\geq 2$ be an even integer, $\chi>0$, $\mfm\in \{0,1\}$, $\zeta \in \mcC^{\alpha_0}_\mfm (\mbT)$ and $w(\zeta) \in C_{\eta;T}\mcC^{\alpha}$ be a mild solution to \eqref{eq:1dRepRemainder} on $[0,T]$. Then there exists a constant $C>0$, independent of $\zeta$, $Z$, $p$ and $\chi$, such that
\begin{equation}\label{eq:1dRemainderAPrioriBound}
\|w_t(\zeta)\|_{L^p}
\leq
\max
\Big\{
(\chi t/4)^{-\frac12},
C(p\vee \chi)^{\gamma}\|Z\|^{\frac{1}{\alpha}}_{C_t\mcC^{\alpha}},\,C(p\vee \chi)^{\frac{p+1}{p+2}}\mfm\,
\Big\},
\end{equation}
where $\gamma=\frac{p(1+\alpha)+2+\alpha}{\alpha(p+2)}$.
\end{theorem}
\begin{proof} Since $t\mapsto \partial_x w_t$ and $t\mapsto (w_t+Z_t)^2\partial_x \rho_{w_t+Z_t}$ are both continuous mappings from $(0,T]$ into $L^\infty(\mbT)$ by Lemma~\ref{lem:spatial_reg}, we may differentiate \eqref{eq:1dRemainderTestEq} with respect to $t$ in order to obtain
\begin{equation*}
\frac{1}{p(p-1)}\frac{\dd}{\dd t} \|w_t\|_{L^p}^p = - \|w_t^{p-2}|\partial_x w_t|^2\|_{L^1}-\chi \langle (w_t+Z_t)^2\partial_x \rho_{w_t+Z_t},w_t^{p-2}\partial_x w_t\rangle,
\end{equation*}
where we used that $p$ is even in the first term on the right hand side. Regarding the second term, since $f\mapsto \partial_x \rho_f$ is linear, we have
\begin{align*}
\langle (w_t+Z_t)^2\partial_x \rho_{w_t+Z_t},w_t^{p-2}\partial_x w_t\rangle
&= \langle \partial_x \rho_{w_t}, w^p_t\partial_x w_t\rangle + \langle \partial_x\rho_{Z_t}, w^p_t \partial_x w_t\rangle\\
&\quad + 2\langle Z_t\partial_x \rho_{w_t+Z_t},w^{p-1}_t\partial_x w_t\rangle + \langle Z_t^2\partial_x\rho_{w_t+Z_t},w^{p-2}_t\partial_x w_t\rangle.
\end{align*}
Integrating the first term by parts, using that $-\partial_{xx}\rho_{w_t}=w_t-\bar{\zeta}$ by Corollary~\ref{cor:1dSolConstantMean} and recalling that we have set $\bar{\zeta} = \mfm$,
\begin{align*}
\langle \partial_x \rho_{w_t}, w^p_t\partial_x w_t\rangle=\frac{1}{p+1}\langle \partial_x \rho_{w_t}, \partial_x w^{p+1}_t\rangle
&=
\frac{1}{p+1}\langle w_t-\mfm,w_t^{p+1}\rangle
\\
&= \frac{1}{p+1} \|w_t^{p+2}\|_{L^1} -\frac{1}{p+1}\langle \mfm,w_t^{p+1}\rangle, 
\end{align*}
Therefore, applying the chain rule in the remaining terms gives
\begin{equation}\label{eq:1dRemLpDiffEquality1}
\begin{split}
\frac{1}{p(p-1)\chi}\frac{\dd}{\dd t} \|w_t\|_{L^p}^p & = - \frac{1}{\chi}\|w_t^{p-2}|\partial_x w_t|^2\|_{L^1} - \frac{1}{p+1} \|w_t^{p+2}\|_{L^1}+ \frac{1}{p+1}\langle 
\mfm,w_t^{p+1}\rangle\\
&\quad - \frac{1}{p+1}\langle \partial_x\rho_{Z_t}, \partial_x w^{p+1}_t\rangle- \frac{2}{p}\langle Z_t\partial_x \rho_{Z_t},\partial_x w^p_t\rangle\\
&\quad - \frac{1}{p-1}\langle Z_t^2\partial_x\rho_{Z_t},\partial_x w^{p-1}_t\rangle- \frac{2}{p}\langle Z_t,\partial_x \rho_{w_t}\partial_x w^p_t\rangle
\\
&\quad -\frac{1}{p-1}\langle Z_t^2,\partial_x\rho_{w_t}\partial_x w^{p-1}_t\rangle.
\end{split}
\end{equation}
We demonstrate that all terms without a definite sign can be bounded by a constant multiple of the two negative terms, 
\begin{equation*}
\frac{1}{\chi}\|w_t^{p-2}|\partial_xw_t|^2\|_{L^1}+\frac{1}{p+1}\|w_t^{p+2}\|_{L^1}=: \frac{1}{\chi}A_t + \frac{1}{p+1}B_t.
\end{equation*}
From Theorems~\ref{th:BesovEmbedding} and~\ref{th:BesovElliptic}, for any $\alpha \in \mbR$ and $p,q \in [1,\infty]$, we have that
\begin{equation}\label{eq:BesovGradElliptic2}
\|\partial_x \rho_w\|_{\mcB^{\alpha}_{p,q}} \lesssim \|w\|_{\mcB^{\alpha-\frac{1}{p}}_{1,q}},\quad \text{ and }\quad \|\partial_x \rho_w\|_{L^\infty} \lesssim \|w\|_{L^1}.
\end{equation}
For any $0<q\leq p+2$, by Jensen's inequality we have
\begin{equation}\label{eq:Jensen1}
\|w_t^q\|_{L^1}\leq B_t^{\frac{q}{p+2}}.
\end{equation}
Furthermore, by applying Cauchy--Schwarz followed by \eqref{eq:Jensen1}, for $ \frac{2}{p}< q \leq p+1$ we have 
\begin{equation}\label{eq:GradientOfPowerBound}
\|\partial_x w^q_t\|_{L^1} = q\| w^{q-1}_t \partial_x w_t\|_{L^1} \leq  q\|w^{2q-p}_t\|^{\frac{1}{2}}_{L^1}\,A_t^{\frac{1}{2}}\leq qB_t^{\frac{2q-p}{2(p+2)}}\,A_t^{\frac{1}{2}}.
\end{equation}
To keep track of dependence on $p$, we also make use of the following inequality which readily follows from Young's inequality for products:
for $\gamma_1,\gamma_2 \in (1,\infty)$ such that $\frac{1}{\gamma_1}+\frac{1}{\gamma_2}=1$, $a,\,b>0$ and $c\geq1$
\begin{equation}\label{eq:YoungPeterAndPaul}
ab < \frac{c^{\gamma_1-1}}{\gamma_1-1} a^{\gamma_1} + \frac{1}{c}b^{\gamma_2}
\end{equation}
(in fact $ab < \frac{c^{\gamma_1-1}}{e(\gamma_1-1)} a^{\gamma_1} + \frac{1}{c}b^{\gamma_2}$).
From now on we let $C>0$,\,$c\geq 1$ be constants, that are independent of $\zeta,\,Z,\,p,$ and $\chi$. Later, we will fix $c\geq 1$ sufficiently large at the end of the proof. If we write $\lesssim$ in an inequality below, the implied proportionality constant is equal to $C>0$ which we take sufficiently large so that the inequality holds.

We work through the terms of \eqref{eq:1dRemLpDiffEquality1} in order. For the third term of \eqref{eq:1dRemLpDiffEquality1}, applying H{\"o}lder, \eqref{eq:Jensen1}, \eqref{eq:YoungPeterAndPaul} and using that $\mfm\in \{0,1\}$, we have,
\begin{equation}\label{eq:FMeanTerm}
\hspace{-1em}	\frac{1}{p+1}|\langle \mfm, w^{p+1}_t\rangle| \leq\frac{\mfm}{p+1} \|w^{p+1}_t\|_{L^1} \leq \frac{\mfm}{p+1}B_t^{\frac{p+1}{p+2}} \leq \frac{\mfm c^{p+1}}{(p+1)^2} + \frac{1}{c(p+1)}B_t.
\end{equation}
For the fourth term of \eqref{eq:1dRemLpDiffEquality1}, we integrate by parts to give
\begin{equation}\label{eq:PartialRhoVPartialW}
\hspace{-1em}\begin{aligned}
\frac{1}{p+1}|\langle \partial_x \rho_{Z_t}, \partial_x w_t^{p+1}\rangle | =\frac{1}{p+1} |\langle Z_t,w_t^{p+1}\rangle|\hspace{0.2em}\overset{\tiny\eqref{eq:BesovDuality}\, \eqref{eq:Jensen1}}{\lesssim}&\frac{1}{p} \|Z\|_{C_t\mcC^{\alpha}} B_t^{\frac{p+1}{p+2}}\\
\overset{\tiny \eqref{eq:YoungPeterAndPaul}}{\leq}\hspace{0.8em}& \frac{c^{p+1}}{p^2} \|Z\|_{C_t\mcC^{\alpha}}^{p+2} + \frac{1}{cp}B_t.
\end{aligned}
\end{equation}
Concerning the two remaining terms of \eqref{eq:1dRemLpDiffEquality1} involving $\partial_x \rho_{Z_t}$, we have for $k=1,2$ 
\begin{equation}\label{eq:VkPartialRhoVPartialW}
\begin{aligned}
\frac{1}{p+1-k}|\langle Z^k_t\partial_x \rho_{Z_t},\partial_x w^{p+1-k}_{t} \rangle | \hspace{1.3em}\leq \hspace{2.3em}& \frac{1}{p+1-k}\| Z^k_t\partial_x \rho_{Z_t}\|_{L^\infty}\, \|\partial_x w^{p+1-k}_{t} \|_{L^1}\\
\overset{\tiny \eqref{eq:BesovProduct},\,\eqref{eq:BesovGradElliptic},\,\eqref{eq:GradientOfPowerBound}}{\lesssim}& \|Z_t\|^{k+1}_{\mcC^\alpha}B_t^{\frac{p+2-2k}{2(p+2)}}A_t^{\frac{1}{2}}\\
\overset{\tiny\eqref{eq:YoungPeterAndPaul}}{ \leq }\hspace{1.9em}& c \|Z_t\|^{2(k+1)}_{\mcC^{\alpha}}\,B_t^{\frac{p+2-2k}{p+2}} + \frac{1}{c} A_t\\
\overset{\tiny\eqref{eq:YoungPeterAndPaul}}{ \lesssim }\hspace{1.9em}& \frac{c^{\frac{p+2-k}{k}}}{p} \|Z\|^{\frac{k+1}{k}(p+2) }_{C_t\mcC^{\alpha}} + \frac{1}{c}B_t + \frac{1}{c} A_t.
\end{aligned}
\end{equation}
(In the case $p=k=2$, note that we do not need to apply~\eqref{eq:YoungPeterAndPaul} in the final line.)
Combining \eqref{eq:PartialRhoVPartialW}, \eqref{eq:VkPartialRhoVPartialW} and noting that $\frac{3}{2}(p+2)$ is the highest power of  $\|Z\|_{C_t\mcC^{\alpha}}$ encountered so far, we have that
\begin{equation}\label{eq:FRhoVTerms}	\begin{aligned}
\hspace{-1.5em}\sum_{k=0}^2 \binom{2}{k} \frac{1}{p+1-k}	|\langle Z^k_t\partial_x \rho_{Z_t},\partial_x w^{p+1-k}_{t} \rangle | &\lesssim
\frac{c^{p+1}}{p}\|Z\|^{\frac{3}{2}(p+2)}_{C_t\mcC^{\alpha}} + \frac{1}{c}B_t +\frac{1}{c}A_t.
\end{aligned}
\end{equation}
For the seventh term of \eqref{eq:1dRemLpDiffEquality1}, we first integrate by parts and apply the triangle inequality to obtain
\begin{equation}\label{eq:MixedTerm1Split}
\frac{1}{p}|\langle Z_t,\partial_x \rho_{w_t}\partial_x w_t^p\rangle |\, \leq \,	\frac{1}{p}|\langle \partial_x Z_t ,w^p_t\partial_x \rho_{w_t}\rangle| + 	\frac{1}{p}|\langle Z_t,w_t^{p+1}\rangle|+\frac{1}{p}|\langle Z_t,w^{p}_t\mfm\rangle|.
\end{equation}
Since $\alpha\in(0,\frac12)$,
using Theorems \ref{th:BesovProduct} and \ref{th:BesovPoincare}, we have
\begin{align}
\|w_t^p\partial_x \rho_{w_t}\|_{\mcB^{1-\alpha}_{1,1}} \hspace{1.6em}\overset{\tiny \eqref{eq:BesovPositiveProductSplit} }{\lesssim}\hspace{1.7em} &\|w_t^p\|_{L^{\frac{1}{1-\alpha}}}\|\partial_x \rho_{w_t}\|_{\mcB^{1-\alpha}_{\frac{1}{\alpha},1}} + \| w^p_t\|_{\mcB^{1-\alpha}_{1,1}}\|\partial_x \rho_{w_t}\|_{L^\infty} \notag \\
\overset{\tiny \eqref{eq:BesovGradElliptic2},\,\eqref{eq:BesovEmbedding},\,\eqref{eq:Besov_Lp_embedding}}{\lesssim}& \|w_t^p\|_{\mcB^{1-\alpha}_{1,1}}\|w_t\|_{L^1}\notag\\
\overset{\tiny \eqref{eq:BesovPoincare}}{\lesssim}\hspace{1.7em}&\left( \|\partial_x w_t^p\|^{1-\alpha}_{L^1}\|w_t^p\|^\alpha_{L^1} + \|w^p_t\|_{L^1}\right)\|w_t\|_{L^1} \notag\\
\overset{\tiny \eqref{eq:Jensen1}}{\lesssim}\hspace{1.6em} &\|\partial_x w_t^p\|^{1-\alpha}_{L^1}B_t^{\frac{p\alpha+1}{p+2}} +  B_t^{\frac{p+1}{p+2}} \notag \\
\overset{\tiny \eqref{eq:GradientOfPowerBound}}{\lesssim}\hspace{1.6em} &p^{1-\alpha}B_t^{\frac{p(1+\alpha)+2}{2(p+2)}}A_t^{\frac{1-\alpha}{2}} + B_t^{\frac{p+1}{p+2}} \label{eq:wPGradRhoWBound}.
\end{align}
Considering the first term of \eqref{eq:MixedTerm1Split}, applying \eqref{eq:wPGradRhoWBound} and then~\eqref{eq:YoungPeterAndPaul} twice with $\gamma_1=\frac{2}{1+\alpha}$ and then with $\gamma_1=\frac{(p+2)(1+\alpha)}{2\alpha}$, and using that $p^{-\alpha}\leq 1$,
\begin{align}
\frac{1}{p}|\langle \partial_x Z_t ,w^p_t\partial_x \rho_{w_t}\rangle| \hspace{0.7em} \overset{\tiny \eqref{eq:BesovDuality}}{\lesssim} \hspace{0.7em}	& \frac{1}{p}\|\partial_x Z_t \|_{\mcC^{\alpha-1}}\|w_t^p\partial_x \rho_{w_t}\|_{\mcB^{1-\alpha}_{1,1}} \notag \\
\overset{\tiny \eqref{eq:wPGradRhoWBound} }{\lesssim}\hspace{0.6em} &\| Z_t \|_{\mcC^{\alpha}} \left(B_t^{\frac{p(1+\alpha)+2}{2(p+2)}}A_t^{\frac{1-\alpha}{2}} + 	\frac{1}{p}B_t^{\frac{p+1}{p+2}} \right) \notag \\
\overset{\tiny \eqref{eq:YoungPeterAndPaul} }{\lesssim}\hspace{0.6em}&
c \|Z\|^{\frac{2}{1+\alpha}}_{C_t\mcC^{\alpha}}B_t^{\frac{p(1+\alpha)+2}{(p+2)(1+\alpha)}}+ \frac{c^{p+1}}{p^2}\|Z\|^{p+2}_{C_t\mcC^{\alpha}} \notag \\
&\quad  +\frac{1}{c}A_t + \frac{1}{pc}B_t
\notag \\
\overset{\tiny \eqref{eq:YoungPeterAndPaul} }{\lesssim} \hspace{0.6em}&
 \frac{c^{\frac{p(1+\alpha)+2+\alpha}{\alpha}} }{p}\|Z\|^{\frac{p+2}{\alpha}}_{C_t\mcC^{\alpha}}  + \frac{p+1}{cp}B_t+\frac{1}{c}A_t. \label{eq:PartialVWPartialW}
\end{align} 
Concerning the second term of \eqref{eq:MixedTerm1Split}, we have
\begin{equation}\label{eq:VWPPlusOne}
\begin{aligned}
\frac{1}{p}|\langle Z_t,w_t^{p+1}\rangle|&\overset{\tiny \eqref{eq:BesovDuality},\,\eqref{eq:Jensen1}}{\lesssim }
\frac1p \|Z\|_{C_t\mcC^{\alpha}}B_t^{\frac{p+1}{p+2}}
\overset{\tiny \eqref{eq:YoungPeterAndPaul}}{\lesssim}\,\frac{c^{p+1}}{p} \|Z\|^{p+2}_{C_t\mcC^{\alpha}}+ \frac{1}{pc}B_t.
\end{aligned}
\end{equation}
For the third term of \eqref{eq:MixedTerm1Split}, by similar estimates and using that $\mfm\in \{0,1\}$ we have,
\begin{equation}\label{eq:VWPm}
    \frac{1}{p}|\langle Z_t,w_t^p \mfm\rangle| \overset{\tiny \eqref{eq:BesovDuality},\,\eqref{eq:Jensen1}}{\lesssim} \frac{1}{p}\|Z\|_{C_t\mcC^\alpha}\mfm B^{\frac{p}{p+2}} \overset{\tiny \eqref{eq:YoungPeterAndPaul}}{\lesssim} \frac{c^p}{p^2}\|Z\|^{p+2}_{C_t\mcC^\alpha} + \frac{c^p}{p^2}\mfm + \frac{1}{cp}B_t
\end{equation}
Regarding the eighth and final term of \eqref{eq:1dRemLpDiffEquality1},
\begin{align}
\frac{1}{p-1}|\langle Z_t^2,\partial_x \rho_{w_t}\partial_x w_t^{p-1}\rangle| \hspace{0.8em}\overset{\tiny \eqref{eq:BesovDuality},\,\eqref{eq:BesovProduct}}{\lesssim} \hspace{0.8em}&\frac{1}{p} \|Z\|^2_{C_t\mcC^{\alpha}}\|\partial_x \rho_{w_t}\|_{L^\infty}\|\partial_x w^{p-1}_t\|_{L^1} \notag \\
\overset{\tiny \eqref{eq:BesovGradElliptic2},\,\eqref{eq:GradientOfPowerBound}}{\lesssim}\hspace{0.6em}&\|Z\|^2_{C_t\mcC^{\alpha}} B_t^{\frac{p}{2(p+2)}}A_t^{\frac{1}{2}} \notag \\
\overset{\tiny \eqref{eq:YoungPeterAndPaul}}{\leq}\hspace{1.5em} &c\|Z\|^4_{C_t\mcC^{\alpha}} B_t^{\frac{p}{p+2}}+\frac{1}{c}A_t \notag \\
\overset{\tiny \eqref{eq:YoungPeterAndPaul}}{\lesssim}\hspace{1.5em} &
\frac{c^{p+1}}{p} \|Z\|^{2(p+2)}_{C_t\mcC^{\alpha}} + \frac{1}{c}B_t + \frac{1}{c}A_t. \label{eq:VSquarePartialWPartialRhoW}
\end{align}
So, since $p+1\leq \frac{p(1+\alpha)+2+\alpha}{\alpha} =:\beta$ and $c\geq 1$,
combining \eqref{eq:PartialVWPartialW}--\eqref{eq:VSquarePartialWPartialRhoW},
\begin{equation}\label{eq:FRhoWTerms}
\sum_{k=1}^2
\binom{2}{k}
\frac{1}{p+1-k}|\langle Z_t^k,\partial_x \rho_{w_t}\partial_x w_t^{p+1-k}\rangle | \lesssim \frac{ c^{\beta}}{p}\|Z\|^{\frac{p+2}{\alpha}}_{C_t\mcC^{\alpha}}+\frac{c^{p}}{p^2}\mfm+\frac{1}{c}B_t +\frac{1}{c}A_t.
\end{equation}
Since $\frac{1}{4(p-1)} \leq \frac{3/4}{p+1}$,
combining
\eqref{eq:1dRemLpDiffEquality1} with \eqref{eq:FMeanTerm}, \eqref{eq:FRhoVTerms}, \eqref{eq:FRhoWTerms} and choosing $c \geq (p\vee \chi)C$ with $C>0$ the sufficiently large implied constant, we obtain that
\begin{align*}
\frac{1}{ \chi p(p-1)}\frac{\dd}{\dd t} \|w_t\|^p_{L^p}\leq&
\max
\Big\{\frac{(C (p\vee \chi))^{\beta}}{p}\|Z\|^{\frac{p+2}{\alpha}}_{C_t\mcC^{\alpha}},\,
\frac{(C (p\vee \chi))^{p+1}\mfm}{p^2}
\Big\}
\\
&-\frac{1}{4(p-1)}\|w^{p+2}_t\|_{L^1}.
\end{align*}
By Jensen's inequality,$	(\|w_t\|^p_{L^p})^{\frac{p+2}{p}}\leq \|w^{p+2}_t\|_{L^1}$, so that for $C>0$ sufficiently large
\begin{equation*}
\frac{\dd}{\dd t} \|w_t\|^p_{L^p}+ \frac{\chi p}{4}\left(\|w_t\|^p_{L^p}\right)^{\frac{p+2}{p}} \leq
\chi\max
\Big\{p(C (p\vee \chi))^{\beta}\|Z\|^{\frac{p+2}{\alpha}}_{C_t\mcC^{\alpha}},
\,(C (p\vee \chi))^{p+1}\mfm
\Big\}.
\end{equation*}
Applying Lemma \ref{lem:ODEComparison} with $f(t)= \|w_t\|^p_{L^p}$, $\lambda = \frac{p+2}{p}$ and possibly increasing $C>0$,
\begin{equation*}
\|w_t(\zeta)\|^p_{L^p} \leq
\max
\Big\{
(\chi t/4)^{-\frac p2},
(C (p\vee \chi))^{\beta \frac{p}{p+2}}\|Z\|^{\frac{p}{\alpha}}_{C_t\mcC^{\alpha}},\,(C (p\vee \chi))^{\frac{p(p+1)}{p+2}}\mfm/p^{\frac{p}{p+2}}\,
\Big\},
\end{equation*}
for every $t \in (0,T]$ and $\zeta \in \mcC^{\alpha_0}$. The stated bound, \eqref{eq:1dRemainderAPrioriBound}, then follows after taking the $p^{\text{th}}$ root on both sides and noticing that $\gamma=\frac\beta{p+2}\leq \frac{4+3\alpha}{4\alpha}$.
\end{proof}
\begin{remark}\label{rem:PLimit}
Observe that the bound \eqref{eq:1dRemainderAPrioriBound} is trivial in the limit $p\rightarrow \infty$, since $\gamma\geq 1$ and so $(p\vee \chi)^{\gamma}\rightarrow \infty$. This is in contrast to \cite{moinat_weber_20_RD}, where an a priori bound of the same form in $L^\infty(\mbT)$ is shown directly for stochastic reaction-diffusions.
\end{remark}
\begin{remark}\label{rem:TailBound}
In \cite{moinat_weber_20_RD}, the equivalent of \eqref{eq:1dRemainderAPrioriBound}, in the context of stochastic reaction-diffusion equations, is established with $\|Z\|^{1/\alpha}_{C_T\mcC^{\alpha}}$ replaced by $\|Z\|^{1/(1+\alpha)}_{C_T\mcC^{\alpha}}$. This leads to lighter than Gaussian tail bounds in the case of the reaction-diffusion equation, again contrasted with the heavier than Gaussian tails that we are able to establish for \eqref{eq:1dRepEquation}, see Theorem \ref{th:1dexpErgodic}.
\end{remark}
\begin{remark}\label{rem:Extensions}
While a mild generalisation of this analysis to $f(u)$, smooth and asymptotically quadratic  ($c_1u^2<f(u)<c_2u^2$), would likely be possible, more natural generalisations such as $f(u)=|u|$ or $f(u)=u^{m-1}$ for $m\geq 5$ and odd, seem to present significant challenges. In the case of $f(u)=|u|$, since the testing argument we follow only applies to the remainder $w=u-Z$, the nonlinearity becomes $$\partial_x(|w_t+Z_t|\partial_x\rho_{w_t+Z_t}) = \partial_x (|w_t+Z_t|)\partial_x\rho_{w_t+Z_t} - |w_t+Z_t|(w_T-\bar{w}+Z_t).$$ When testing with $w_t^{p-1}$ this no longer produces the necessary damping term $-\|w_t^{p+2}\|_{L^1}$ and so it is not directly clear how one could proceed in this case. Regarding higher polynomial powers, it appears that controlling the transport term causes a problem for $m>3$. In particular estimating the equivalent of \eqref{eq:MixedTerm1Split} in the same way, leads us to control $|\langle \partial_x Z , w^{p+m-2}\partial_x \rho_{w}\rangle |$. Repeating similar estimates as in \eqref{eq:wPGradRhoWBound} and \eqref{eq:PartialVWPartialW} lead to the condition $m<\frac{3-\alpha}{1-\alpha}$. For $\alpha \in (0,1/2)$ this restricts us to $m<5$. Letting $\alpha\rightarrow 1$ would allow $m\rightarrow \infty$. Informally speaking, integrability of the solution resulting from the transport term appears to be linked to regularity of the noise - higher regularity leads to higher integrability. This was also observed in the case of stochastic reaction-diffusion equations in \cite{moinat_weber_20_RD}. Note also that with $\alpha\geq 1$ one could apply the testing argument directly to $u$, rather than working with the remainder. In \cite{lai_xiao_17} it was in fact shown that for the same model but with $\xi\equiv 0$, global existence and convergence to equilibrium holds for all $m>0$. In the case of more regular noise $v\in C_T\mcC^{1}(\mbT)$, we would therefore expect at least an analogue of Theorem \ref{th:1dRemainderAPriori} to hold directly for $u$ also for all $m>0$.
\end{remark}
\begin{proposition}\label{prop:RemainderGWP}
Let $\mfm\in \{0,1\}$, $\zeta \in \mcC^{\alpha_0}_\mfm(\mbT)$ and $Z \in C_T\mcC^{\alpha}_0$. Then there exists a unique mild solution $w(\zeta)\in C_{\eta;T}\mcC^{\alpha}$ to~\eqref{eq:1dRepRemainder}. Furthermore, the map,
\begin{equation}\label{eq:1dRepSolMap}
\begin{aligned}
\msS : \mcC^{\alpha_0}_\mfm(\mbT)\times C_T \mcC^{\alpha}_0(\mbT) &\rightarrow C_{\eta;T} \mcC^{\alpha}_\mfm(\mbT)\\
(\zeta,Z)&\mapsto u(\zeta):=w(\zeta)+Z,
\end{aligned}
\end{equation}
is jointly, locally Lipschitz.
\end{proposition}
\begin{proof}
We recall from Theorem \ref{th:1dRemainderLWP} that there exists a $T_*\in (0,1)$, depending only on $\|\zeta\|_{\mcC^{\alpha_0}}$ and $\|Z\|_{C_T\mcC^\alpha}$, such that a mild solution $w \in C_{\eta;T_*}\mcC^{\alpha}$ exists to $\eqref{eq:1dRepRemainder}$. Without loss of generality let us assume $T>T_*$ and that we fix an even $p \in (-\frac{1}{\alpha_0},\infty)$ so that $L^p(\mbT)\hookrightarrow\mcC^{\alpha_0}(\mbT)$. In this case, it is clear that we can extend the solution for a positive time of existence so long as $\|w_t\|_{L^p}$ remains finite. However, by Theorem \ref{th:1dRemainderAPriori}, $\|w_t\|_{L^p}$ is bounded above by a function of $t$ independent of the initial data and so we may continue the solution to all of $[0,T]$.
From Corollary \ref{cor:1dSolConstantMean}, for $u_t:= w_t+Z_t$, we have $\bar{u}_t = \bar{\zeta}+\bar{Z}_t$, for all $t\in (0,T]$. Similarly, for all $t>0$, $\|u_t\|_{\mcC^{\alpha}}\leq \|w_t\|_{\mcC^{\alpha}}+\|Z_t\|_{\mcC^{\alpha}}$. Hence the solution map~\eqref{eq:1dRepSolMap} is well-defined and $\|u\|_{C_{\eta,T}\mcC^\alpha}$ depends only on $\|Z\|_{C_T\mcC^\alpha}$.

To continue the proof, we state the following

\begin{lemma}\label{lem:short_time_Lip}
Consider $\mfR>0$
and define the set
\[
D_\mfR := \{(\zeta,Z)\in\mcC^{\alpha_0}_\mfm\times C_T\mcC^\alpha\,:\,\|\zeta\|_{\mcC^{\alpha_0}} + \|Z\|_{C_T\mcC^\alpha} < \mfR \}.
\]
Then for $T_*=T_*(\mfR)>0$ sufficiently small,
$\msS\colon D_\mfR \to C_\eta((0,2T_*],\mcC^\alpha_\mfm)$
is $K$-Lipschitz where $K=K(\mfR)>0$.
\end{lemma}

\begin{proof}
Let $(\zeta,Z),(\tilde{\zeta},\tilde{Z})\in D_\mfR$ and consider the corresponding solutions
\begin{align*}
u_t &= e^{t\Delta}\zeta + \int_0^t e^{(t-s)\Delta}\chi\partial_x(u^2_s\partial_x\rho_{u_s})\,\dd s + Z_t,\\
\tilde{u}_t &= e^{t\Delta}\tilde{\zeta} + \int_0^t e^{(t-s)\Delta}\chi\partial_x(\tilde{u}^2_s\partial_x\rho_{\tilde{u}_s})\,\dd s + \tilde{Z}_t.
\end{align*}
From Theorem \ref{th:1dRemainderLWP},
there exists $T_*(\mfR)>0$ such that
$\|u\|_{C_{\eta;2T_*}\mcC^{\alpha}}\vee \, \|\tilde{u}\|_{C_{\eta;2T_\star}\mcC^{\alpha}}\leq 2$.
For $t \in (0,2T_*]$, using Theorem \ref{th:BesovHeatFlow} and similar bounds as in the proof of Theorem \ref{th:1dRemainderLWP} we see that
\begin{align*}
\|u_t-\tilde{u}_t\|_{\mcC^{\alpha}} &\lesssim t^{-\frac{\alpha-\alpha_0}{2}}\|\zeta-\tilde{\zeta}\|_{\mcC^{\alpha_0}}  +  t^{\frac{1}{2}-3\eta} \|u-\tilde{u}\|_{C_{\eta;t}\mcC^{\alpha}}+ \|Z_t-\tilde{Z}_t\|_{\mcC^{\alpha}},
\end{align*}
where the proportionality constant does not depend on $\zeta,\tilde\zeta,Z,\tilde Z$.
So multiplying through by $t^\eta$ and taking the supremum over $t\in(0,2T_*]$ we have that
\begin{equation*}
\|u-\tilde{u}\|_{C_{\eta;2T_*}\mcC^{\alpha}} \lesssim \|\zeta-\tilde{\zeta}\|_{\mcC^{\alpha_0}} + T_*^{\frac{1}{2}-2\eta}	\|u-\tilde{u}\|_{C_{\eta;S}\mcC^{\alpha}} + \|Z-\tilde{Z}\|_{C_T\mcC^{\alpha}}.
\end{equation*}
By lowering $T_*$ further,
we obtain
\begin{equation*}
\|u-\tilde{u}\|_{C_{\eta;2T_*}\mcC^{\alpha}} \lesssim\|\zeta-\tilde{\zeta}\|_{\mcC^{\alpha_0}} +\|Z-\tilde{Z}\|_{C_T\mcC^{\alpha}}.\qedhere
\end{equation*}
\end{proof}

We now prove that $\msS$ is jointly locally Lipschitz.
Consider $\mfR>1$, $D_\mfR$, and $T_*(\mfR)>0$ as in Lemma~\ref{lem:short_time_Lip}.
Then $\msS$ is $K(\mfR)$-Lipschitz from $D_\mfR$ to $C_\eta((0,2T_*],\mcC^{\alpha})$.
In particular, the map $(\zeta,Z) \mapsto u|_{[T_*,2T_*]} \in C([T_*,2T_*],\mcC^\alpha)$ is $K(\mfR)$-Lipschitz.
If $T\leq 2T_*$, then we are done.
Hence, suppose $T>2T_*$.
We will prove that, for $\bar T_*(\mfR)\in (0,T_*]$ sufficiently small,
$u|_{[T_*+\bar T_*,T]}\in C([\bar T_*,T],\mcC^\alpha)$ is a Lipschitz function of $(\zeta,Z)$, from which the conclusion will follow.

To this end, consider $(\zeta,Z)\in D_\mfR$ and let $u$ be the corresponding solution.
By the a priori estimate, for $p\geq 2$ sufficiently large 
$\|u_t\|_{\mcC^{\alpha_0}}\lesssim_{\alpha_0,p} \|u_t\|_{L^p}\lesssim t^{-1/2}\vee \mfR^{1/\alpha}$ for all $t>0$.
Therefore, there exists $\bar\mfR(\mfR)>0$
such that $\|u_t\|_{\mcC^{\alpha_0}}\leq\bar\mfR$ for all $t\in[T_*,T]$.

Let $\bar T_*(\bar\mfR)>0$ be the constant from Lemma~\ref{lem:short_time_Lip}.
Without loss of generality,
we can assume that $\bar T_*\leq T_*$
and that $T=T_*+N \bar T_*$ for some integer $N\geq 2$.
By the above remark, for all $1\leq n < N$, $u|_{[T_*+n \bar T_*,T_*+(n+1)\bar T_*]}\in C([T_*+n\bar T_*,T_*+(n+1)\bar T_*],\mcC^\alpha)$ is a $K(\mfR)$-Lipschitz function of $(u_{T_*+(n-1)\bar T_*},Z)\in \mcC^{\alpha_0}\times C_T\mcC^\alpha$.
Combining these estimates together, it follows that
$u|_{[T_*+\bar T_*,T]}\in C([\bar T_*,T],\mcC^\alpha)$ is a Lipschitz function of $(\zeta,Z)$.
\end{proof}
We conclude this section with the proof of Theorem \ref{th:1dGWP}.
\begin{proof}[Proof of Theorem \ref{th:1dGWP}] From Theorem \ref{th:MarkovSHESpaceTimeRegular}, for any $T>0$ and $\alpha<\frac12$, $v_{0,\,\cdot\,}\in C_T\mcC_0^{\alpha}$ (in fact $v_{0,\,\cdot\,} \in \cap_{\kappa \in [0,1/2)}\mcC^{\kappa}_T\mcC_0^{\alpha-2\kappa}$). Hence Proposition~\ref{prop:RemainderGWP} implies that $\mbP$-a.s. there exists a unique mild solution $u(\zeta):=w(\zeta)+v_{0,\,\cdot\,}=\msS(\zeta,v_{0,\,\cdot\,})$ to~\eqref{eq:1dRepEquation} on $[0,T]$.
\end{proof}
\section{Invariant Measure}\label{sec:InvariantMeasures}
Having established global well-posedness of the SPDE \eqref{eq:1dRepEquation} we turn to study its long time behaviour. In this section we prove Theorem~\ref{th:1dexpErgodic} which implies
existence and uniqueness of the invariant measure and demonstrates exponential convergence of the adjoint semi-group to the invariant measure in the topology of total variation.

By Remark~\ref{rem:chi}, appropriately adjusting the spatial mean $\mfm$ of the initial condition $\zeta$, it suffices to prove Theorem~\ref{th:1dexpErgodic} in the case $\chi=1$.
We thus assume $\chi=1$ throughout this section.

\begin{definition}\label{def:u}
Let $v_{t_0,\,\cdot\,}$
denote the mild solution to the stochastic heat equation as in Definition~\ref{def:SHE}
and denote $v_t:=v_{0,t}$.
For $\zeta\in\mcC^{\alpha_0}$, the process $u(\zeta):=w(\zeta)+v_{\cdot}=\msS(\zeta,v_{\cdot})$, with $\msS$ as defined in Proposition~\ref{prop:RemainderGWP},
is called the mild solution on $[0,T]$ to~\eqref{eq:1dRepEquation} with initial condition $\zeta$.

We define the semi-group associated to \eqref{eq:1dRepEquation} by setting, for all $t\geq 0$,
\begin{equation*}
P_t\Phi(\zeta):= \mbE\left[\Phi(u_t(\zeta))\right], \quad \text{ for all } \Phi \in \mcB_b(\mcC^{\alpha_0}(\mbT)).
\end{equation*}
\end{definition}
We first show that $(u_t)_{t>0}$ is a Markov process with Feller semi-group $(P_t)_{t>0}$ and that for each $\zeta \in \mcC^{\alpha_0}_\mfm(\mbT)$, there exists a measure $\nu_\zeta \in \mcP(\mcC^{\alpha_0}_\mfm)$, invariant for $(P_t)_{t>0}$.
\begin{lemma}\label{lem:MarkovDecompose}
Let $u$ be a mild solution to \eqref{eq:1dRepEquation}. Then for every $t \in [0,T]$ and $h \in (0,T-t)$ we have the identity
\begin{equation}\label{eq:1dMarkovDecompose}
u_{t+h} = \tilde{w}_{t,t+h}+v_{t,t+h} ,
\end{equation}
where $\tilde{w}_{t,t+h}$ solves
\begin{equation}\label{eq:1dwTildeEquation2}
\tilde{w}_{t,t+h} = e^{h\Delta}u_t + \int_0^h e^{(h-r)\Delta}\partial_x ((\tilde{w}_{t,t+r}+v_{t,t+r})^2\partial_x \rho_{\tilde{w}_{t,t+r}+v_{t,t+r}})\,\dd r.
\end{equation}
\end{lemma}
\begin{proof}
A simple consequence of the heat semi-group property.
\end{proof}
\begin{theorem}\label{thm:SolAPriori}
For $\mfm \in \mbR$, and $\zeta \in \mcC^{\alpha_0}_\mfm(\mbT)$, let $u(\zeta)\in C_{\eta;T}\mcC^{\alpha}$ be the unique solution to \eqref{eq:1dRepEquation} as in Definition~\ref{def:u}. Then for every $p\in [1,\infty)$, we have that
\begin{equation}\label{eq:1dSolexpectationApriori}
\sup_{\zeta \in \mcC^{\alpha_0}_\mfm(\mbT)}\sup_{T>0} \sup_{t\in (0,T]} \left(t^{\frac{p}{2}}\wedge 1 \right) \mbE\left[ \|u_t(\zeta)\|^p_{L^{p}} \right] <\infty.
\end{equation} 
\end{theorem}
\begin{proof}
It suffices to consider $T>1$.
Consider first $p\geq 2$ even.
For $t \in (0,1]$, applying \eqref{eq:1dRemainderAPrioriBound} and Theorem \ref{th:MarkovSHESpaceTimeRegular},
\begin{equation}\label{eq:1dSolSmallTimeApriori}
\mbE\left[ \|u_t\|^p_{L^p}\right] \lesssim \mbE\left[ \|w_t\|_{L^p}^p \right] + \mbE\left[ \|v_t\|_{L^p}^p\right] \lesssim t^{-\frac{p}{2}}+1.
\end{equation}
For $t\in(1,T]$ we employ the Markov decomposition \eqref{eq:1dMarkovDecompose} to give
\begin{equation*}
\mbE\left[\|u_t\|_{L^p}^p\right] \lesssim \mbE\left[\|\tilde{w}_{t-1,t}\|_{L^p}^p\right] + \mbE\left[ \|v_{t-1,t}\|_{L^p}^p \right].
\end{equation*}
Then, observing that $\tilde{w}_{t-1,t}$ solves \eqref{eq:1dRepRemainder} with initial condition $u_{t-1}(\zeta)$ and driving path $Z_t=v_{t-1,t}$, by Theorem \ref{th:1dRemainderAPriori} we have
\begin{equation*}
\mbE\left[ \|\tilde{w}_{t-1,t}\|_{L^p}^p \right] \leq C_{p,\mfm} \max\Big\{ \|v_{t-1,\,\cdot\,}\|^{\frac{p}{\alpha}}_{C_{[t-1,t]}\mcC^{\alpha}},\,1\,\Big\}.
\end{equation*}
So, since the law of $v_{t-1,t}$ does not depend on $t>1$, we have
\begin{equation}\label{eq:1dSolLargeTimeAPriori}
\sup_{\zeta \in \mcC^{\alpha_0}_{\mfm}(\mbT)}\sup_{T>0} \sup_{t\in(1,T]} \mbE\left[\|u_t(\zeta)\|_{L^p}^p\right] <\infty.
\end{equation}
Combining \eqref{eq:1dSolSmallTimeApriori} with \eqref{eq:1dSolLargeTimeAPriori} proves \eqref{eq:1dSolexpectationApriori} for even $p\geq 2$. The case of general $p\geq 1$ follows by Jensen's inequality.
\end{proof}
\subsection{Existence of Invariant Measures}
We use the decomposition \eqref{eq:1dMarkovDecompose} to show that $t\mapsto u_t$ defines a Markov process.
\begin{theorem}
Let $\zeta \in \mcC^{\alpha_0}(\mbT)$ and $u_t(\zeta)$ as in Definition~\ref{def:u}. Then for any $\Phi \in \mcB_b(\mcC^{\alpha_0}(\mbT))$
\begin{equation*}
\mbE\left[\Phi(u_{t+h}(\zeta))\,\big|\, \mcF_t \right] = P_h(\Phi(u_{t}(\zeta)).
\end{equation*}
In particular, $t\mapsto u_t$ is a Markov process with associated Markov semi-group $(P_t)_{t>0}$.
\end{theorem}
\begin{proof}
Using the Markov decomposition \eqref{eq:1dMarkovDecompose} we see that
\begin{equation*}
\Phi(u_{t+h}(\zeta)) = \Phi(\tilde{w}_{t,t+h}(u_t(\zeta))+v_{t,t+h})=\Phi(u_{t,t+h}(u_t(\zeta))),
\end{equation*}
where $\tilde{w}_{t,t+h}(u_t(\zeta))$ solves \eqref{eq:1dwTildeEquation2}. We define
\begin{equation*}
\psi(u_t(\zeta),v_{t,t+h}) := \Phi(\tilde{w}_{t,t+h}(u_t(\zeta))+v_{t,t+h})
\end{equation*}
and then since $u_t(\zeta)$ is $\mcF_t$ measurable and $v_{t,t+h}\perp \mcF_t$, applying \cite[Prop.~1.12]{daprato_zabczyk_14}, we have that
\begin{align*}
\mbE\left[\Phi(u_{t+h}(\zeta))\,\big|\, \mcF_t \right] = \mbE[\Phi(u_{t,t+h}(u_t(\zeta)))] = P_h(\Phi(u_t(\zeta)))
\end{align*}
and so the claim is shown.
\end{proof}
The following lemma shows that $(P_t)_{t>0}$ is a Feller semi-group in the sense of \cite[Sec. 3.1]{daprato_zabczyk_96}.
\begin{lemma}\label{lem:FellerSemiGroup}
For any $\Phi \in \mcC_b(\mcC^{\alpha_0}(\mbT))$ and $t> 0$ we have that $P_t \Phi \in \mcC_b(\mcC^{\alpha_0}(\mbT))$ and for any $\zeta\in \mcC^{\alpha_0}$, $\lim_{t\rightarrow 0}|P_t\Phi(\zeta) -\Phi(\zeta)|=0$.
\end{lemma}
\begin{proof}
By
Proposition~\ref{prop:RemainderGWP},
$\mcC^{\alpha_0}(\mbT)\ni\zeta\mapsto u_t(\zeta)\in\mcC^{\alpha}(\mbT)$ is $\mbP$-a.s. continuous for every $t>0$.
The fact that $P_t\Phi \in \mcC_b(\mcC^{\alpha_0}(\mbT))$ thus follows from the dominated convergence theorem. The fact that $\lim_{t\searrow0}|P_t\Phi(\zeta)-\Phi(\zeta)|= 0$ for all $\zeta \in \mcC^{\alpha_0},\,\Phi\in \mcC_b(\mcC^{\alpha_0}(\mbT))$, similarly follows the dominated convergence theorem and~\eqref{eq:conv_to_zeta}.
\end{proof}
We are now in a position to show the existence of an invariant measure, $\nu_\zeta \in \mcP(\mcC^{\alpha_0}_\mfm)$ for every $\zeta \in \mcC^{\alpha_0}_\mfm(\mbT)$.
\begin{theorem}\label{thm:existence_invar_measures}
Let $\mfm \in \mbR$. Then for every $\zeta \in \mcC^{\alpha_0}_\mfm(\mbT)$, there exists a measure $\nu_\zeta \in \mcP(\mcC^{\alpha_0}_\mfm(\mbT))$ and an increasing sequence of times $t_k \nearrow \infty$ such that,
\begin{equation*}
\frac{1}{t_k}\int_0^{t_k} P^*_r \,\delta_\zeta \,\dd r \rightharpoonup \nu_\zeta, \quad \text{ as } k\rightarrow \infty.
\end{equation*}
Furthermore $\nu_\zeta$ is invariant for $(P_t)_{t>0}$.
\end{theorem}
\begin{proof}
By the compact embedding $L^p\hookrightarrow \mcC^{\alpha_0}$ for sufficiently large $p\geq 1$,
it follows from Theorem~\ref{thm:SolAPriori} that the family of measures $\{P^*_t \delta_\zeta\}_{t\geq 1, \zeta\in\mcC^{\alpha_0}(\mbT)}$
is tight.
In particular, for every sequence of times $(t_k)_{k>0}$, $\left(R_{t_k}^* \delta_\zeta\right)_{k>0} := \left(\frac{1}{t_k}\int_0^{t_k} P^*_r \,\delta_\zeta \,\dd r\right)_{k>0}$ is a tight family of measures. Applying Prokhorov's theorem gives the existence of a sequence $(t_k)_{k>0}$ for which the necessary weak convergence holds. From \cite[Thm. 3.1.1]{daprato_zabczyk_96} we see that the limit measure is necessarily invariant for $(P_t)_{t> 0}$.
\end{proof}
\subsection{Strong Feller Property}\label{subsec:StrongFeller}
We prove that $(P_t)_{t>0}$ possesses the strong Feller property closely following the method employed in \cite{tsatsoulis_weber_18}. The main step is to establish a Bismut--Elworthy--Li (BEL) type formula, similar to that given in \cite[Thm.~5.5]{tsatsoulis_weber_18}.

For technical reasons in this section we replace \eqref{eq:exponentCriteria} with the assumption that $\alpha_0\in \left(-\frac{1}{3},0\right)$, $\alpha \in \left(0,\alpha_0+\frac{1}{3}\right)$ and $\eta >0$ such that,
\begin{equation}\label{eq:exponentCriteria2}
\frac{\alpha-\alpha_0}{2}<\eta< \frac{1}{6}.
\end{equation}
We note that all previous results also hold for this more restrictive parameter range and this restriction does not affect the main result as stated, see the proof of Theorem \ref{th:1dexpErgodic} at the conclusion of Section \ref{subsec:ExponentialErgodicity}.
Let $(e_m)_{m\in \mbZ}$, $e_m(x)=e^{i2\pi mx}$, be the usual Fourier basis elements of $L^2(\mbT)$ and $(\Delta_k)_{k\geq -1}$ be the Littlewood--Paley projection operators, see Appendix \ref{app:HolderBesovSpaces} for more details.
For $\varepsilon\in (0,1)$, we define $\Pi_\varepsilon(L^2(\mbT))$ as the space of real functions spanned by $\left\{ (e_m)_{|m|<\frac{1}{\varepsilon}}\right\}$ and
\[
\hat{\Pi}_\varepsilon := \sum_{ -1\leq k \leq -\log_{2}(9\varepsilon)} \Delta_k
\colon L^2(\mbT) \to \Pi_\eps(L^2(\mbT^d)).
\]
Observe that there exists $\ell_\eps\colon \mbZ \to [0,1]$ with $\ell_\eps(m)=0$ for $|m|\geq \frac1\eps$ such that $\mcF(\hat\Pi_\eps f)(m) = \ell_\eps(m) \mcF f(m)$ where $\mcF$ is the Fourier transform.
Furthermore (see e.g.~\cite[p.~1213]{tsatsoulis_weber_18}),
\begin{enumerate}[label=\roman*)]
\item \label{ass:ProjectionProperty1} $\displaystyle \|\hat{\Pi}_\varepsilon\|_{\text{Op}(\mcC^\beta;\,\mcC^{\beta})} :=\sup_{\|u\|_{\mcC^{\beta}}\leq1 }\|\hat{\Pi}_\varepsilon u\|_{\mcC^{\beta}} \leq 1$ for all $\eps\in(0,1)$ and $\beta \in \mbR$,
\item \label{ass:ProjectionProperty2} for every $\beta\in\mbR$ and $\delta>0$, there exists $C>0$ such that for all $\eps\in(0,1)$
\begin{equation*}
\|\hat{\Pi}_\varepsilon f-f\|_{\mcC^{\beta-\delta}} \leq C \varepsilon^\delta\|f\|_{\mcC^{\beta}}.
\end{equation*}
\end{enumerate}

We fix for the rest of the subsection $\delta>0$
\begin{equation}\label{eq:ApproxDeltaCriteria}
\alpha+\delta<\frac12,\quad
\eta -\frac{\alpha-\alpha_0+\delta}{2}>0,\quad\frac{1-\delta}{2}-3\eta >0. 
\end{equation}
Such a $\delta$ exists due to~\eqref{eq:exponentCriteria2}.

Consider now $Z\in C_T\mcC^{\alpha+\delta}$
and let $Z_{\eps;t} := \hat\Pi_\eps Z_t$.
Consider likewise $\zeta\in\mcC^{\alpha_0}$ and let $\zeta_\eps:=\hat\Pi_\eps\zeta$.
We define a smooth approximation to \eqref{eq:1dRepEquation}, with $\chi=1$ and deterministic noise, by
\begin{equation}\label{eq:1dApproximateEquation}
u_{\eps;t} = e^{t\Delta}\zeta_\eps + \int_0^t e^{(t-s)\Delta} \partial_x\hat\Pi_\eps(u^2_{\eps;s}\partial_x\rho_{u_{\eps;s}})\dd s + Z_{\eps;t},
\end{equation}
where, as usual,
\[
-\partial_{xx} \rho_{ u_{\varepsilon;t}} =u_{\varepsilon;t}-\bar{u}_\varepsilon, \text{ on }\mbT.
\]

\begin{remark}
We will later set, right before Theorem~\ref{th:BELFormula}, $Z_\eps=\hat\Pi_\eps v$, where $v$ is the SHE.
\end{remark}
\begin{theorem}\label{th:1dApproxEquation}
Let $T>0$, $\zeta \in \mcC^{\alpha_0}_\mfm(\mbT)$, and $Z\in C_T\mcC^{\alpha+\delta}$.
For every $\bar\eps>0$,
there exists $\eps_0(\bar\eps,T,\|\zeta\|_{\mcC^{\alpha_0}}+\|Z\|_{C_T\mcC^{\alpha+\delta}})>0$
such that, for every $\eps\in(0,\eps_0)$
there exists a unique solution $u_\varepsilon \in C_{T}\Pi_\eps(L^2(\T))$ to~\eqref{eq:1dApproximateEquation}. Furthermore
\begin{equation}\label{eq:ApproxUConverge} \|u_{t}-u_{\varepsilon;t}\|_{C_{\eta;T}\mcC^\alpha} \leq \bar\eps.
\end{equation}
\end{theorem}
\begin{proof}
It follows from Properties \ref{ass:ProjectionProperty1} and \ref{ass:ProjectionProperty2} of the projection and the same steps as in the proof of Theorem \ref{th:1dRemainderLWP} that for every $\varepsilon>0$ there exists $T_\eps(\|\zeta_\eps\|_{\mcC^{\alpha_0}}+\|Z_\eps\|_{C_T\mcC^\alpha})>0$,
such that $u_\varepsilon \in C_{\eta;T_\eps}\mcC^\alpha$ solves~\eqref{eq:1dApproximateEquation}
and satisfies $\|u_\varepsilon\|_{C_{\eta;T_\eps}\mcC^\alpha} \leq 1$.
The fact that $u_\eps\in C_{T_\eps}\Pi_\eps(L^2(\mbT))$ follows from the same argument as the end of Theorem~\ref{th:1dRemainderLWP} since $\zeta_\eps\in C^\infty$. 

To continue, we state the following

\begin{lemma}\label{lem:u_approx_u}
Consider $\mfR>0$
and define $D_\mfR$ as in Lemma~\ref{lem:short_time_Lip}.
There exists $T_*(\mfR)>0$ such that,
if $(\zeta_\eps,Z_\eps),(\tilde{\zeta},\tilde{Z})\in D_\mfR$, then
\begin{equation*}
\|\tilde{u}_t-u_{\varepsilon;t}\|_{C_{\eta;T_*}\mcC^{\alpha}} \leq  \|\zeta_\eps-\tilde{\zeta}\|_{\mcC^{\alpha_0-\delta}} + \varepsilon^{\delta} +\|\tilde{ Z}-Z_{\varepsilon}\|_{C_T\mcC^{\alpha}},
\end{equation*}
where $\tilde{u}=\msS(\tilde{\zeta},\tilde{Z})$ with $\msS$ as in Proposition~\ref{prop:RemainderGWP}.
\end{lemma}

\begin{proof}
There exists $T_*(\mfR)>0$ such that
$\|\tilde{u}\|_{C_{\eta;2T_*}\mcC^{\alpha}}+\|u_\eps\|_{C_{\eta;2T_\star}\mcC^{\alpha}}\leq 2$.
For $t \in (0,2T_*]$,
\begin{equation*}
\begin{aligned}
\|\tilde{u}_t-u_{\varepsilon;t} \|_{\mcC^{\alpha}} &\leq \|e^{t\Delta}(\tilde{\zeta}-\zeta_\eps
)\|_{\mcC^{\alpha}} + \int_0^t \|e^{(t-s)\Delta} \partial_x (\tilde{ u}^2_s\partial_x \rho_{\bar u_s} - \hat{\Pi}_\varepsilon (u^{2}_{\varepsilon;s} \partial_x\rho_{u_{\varepsilon;s}}))\|_{\mcC^{\alpha}}\,\dd s\\
&\quad + \|\tilde{ Z}_t-Z_{\varepsilon;t}\|_{\mcC^{\alpha}}.
\end{aligned}
\end{equation*}
Then
\begin{equation*}
\|e^{t\Delta}(\tilde{\zeta}-\zeta_\eps)
\|_{\mcC^{\alpha}}
\lesssim t^{-\frac{\alpha-\alpha_0+\delta}{2}}\|\tilde{\zeta}-\zeta_\eps \|_{\mcC^{\alpha_0-\delta}},
\end{equation*}
and applying the triangle inequality and using Property \ref{ass:ProjectionProperty2} of the projection,
\begin{align*}
\|\tilde{u}_s^2\partial_x \rho_{\bar u_s} - \hat{\Pi}_\varepsilon(u^2_{\varepsilon;s}\partial_x \rho_{u_{\varepsilon;s}})\|_{\mcC^{\alpha-\delta}}
&\leq \|\tilde{u}_s^2\partial_x \rho_{\tilde{u}_s} -u^2_{\varepsilon;s}\partial_x \rho_{u_{\varepsilon;s}} \|_{\mcC^{\alpha}}
\\
&\qquad
+ \|u^2_{\varepsilon;s}\partial_x \rho_{u_{\varepsilon;s}}-\hat{\Pi}_\varepsilon(u^2_{\varepsilon;s}\partial_x \rho_{u_{\varepsilon;s}})\|_{\mcC^{\alpha-\delta}}\\
&\lesssim s^{-2\eta}\|\tilde{u}_s-u_{\varepsilon;s}\|_{\mcC^{\alpha}} + \varepsilon^{\delta}s^{-3\eta}\\
&\leq s^{-3\eta}\|\tilde{u}-u_\eps\|_{C_{\eta;s}\mcC^\alpha} +\varepsilon^{\delta}s^{-3\eta}.
\end{align*}
Therefore, using that $\eta<\frac{1}{6}$, we may integrate $\|e^{(t-s)\Delta} \partial_x (\tilde{ u}^2_s\partial_x \rho_{\bar u_s} - \hat{\Pi}_\varepsilon (u^{2}_{\varepsilon;s} \partial_x\rho_{u_{\varepsilon;s}}))\|_{\mcC^{\alpha}}$ from $0$ to $t$ and take suprema to obtain
\begin{align*}
\sup_{t \in (0,T_*]}t^{\eta}\|\tilde{u}_s-u_{\varepsilon;s}\|_{\mcC^{\alpha}}
&\lesssim T_*^{\eta-\frac{\alpha-\alpha_0-\delta}{2}}\|\tilde{\zeta}-\zeta_\eps\|_{\mcC^{\alpha_0-\delta}} + T_*^{\frac{1-\delta}{2}-3\eta} \sup_{t \in (0,T_*]}t^{\eta}\|\tilde{u}_t-u_{\varepsilon;t}\|_{\mcC^{\alpha}}\\
&\quad\quad + T_*^{\frac{1-\delta}{2}-2\eta}\varepsilon^{\delta} 
+T_*^{\eta}\|\tilde{Z}-Z_{\varepsilon}\|_{C_T\mcC^{\alpha}}.
\end{align*}
After potentially lowering $T_*(\mfR)>0$ further, we obtain
\begin{equation*}
\sup_{t \in (0,T_*]}t^\eta\|\tilde{u}_t-u_{\varepsilon;t}\|_{\mcC^{\alpha}} \leq  \|\tilde{\zeta}-\zeta_\eps\|_{\mcC^{\alpha_0-\delta}} + \varepsilon^{\delta} 
+\|\tilde{Z}-Z_{\varepsilon}\|_{C_T\mcC^{\alpha}} .\qedhere
\end{equation*}
\end{proof}

To conclude the proof of Theorem~\ref{th:1dApproxEquation}, observe that $\|\zeta-\Pi_\eps\zeta\|_{\mcC^{\alpha_0-\delta}} \lesssim \eps^\delta\|\zeta\|_{\mcC^{\alpha_0}}$
and
$\|Z_t-Z_{\eps;t}\|_{\mcC^{\alpha}} \lesssim \eps^\delta\|Z_t\|_{\mcC^{\alpha+\delta}}$
by Property~\ref{ass:ProjectionProperty2}.
By the a priori estimate, for every $\mfR>0$ there exists $T_*(\mfR)>0$ as in Lemma~\ref{lem:u_approx_u} and
$\bar\mfR(\mfR)>0$
such that $\|u_t\|_{\mcC^{\alpha_0}}\leq \bar \mfR$ for all $t\in[T_*,T]$.
Iteratively applying Lemma~\ref{lem:u_approx_u} as in the proof of Proposition~\ref{prop:RemainderGWP}, there exists $\eps_0(\mfR)>0$ sufficiently small
such that $\|u_{\eps;t}\|_{\mcC^\alpha} \leq 2\bar\mfR$
on $[\bar T_*,T]$ for all $\eps<\eps_0$,
from which~\eqref{eq:ApproxUConverge} follows after another iterative application of Lemma~\ref{lem:u_approx_u}.
\end{proof}
By Theorem~\ref{th:1dApproxEquation}, for every $\mfR>0$ and $T>0$, there exists $\eps_0(\mfR,T)>0$ sufficiently small such that for all $\eps\in(0,\eps_0)$ the solution map to~\eqref{eq:1dApproximateEquation}
\begin{equation}\label{eq:1dRepApproxSolMap}
\begin{aligned}
\msS^\varepsilon \colon A_\mfR &\to C_{T}\Pi_\eps(L^2(\mbT))\\
\left(\zeta,Z\right)&\mapsto u_{\varepsilon},
\end{aligned}
\end{equation}
is well-defined
and $\|u_\eps-u\|_{C_{\eta;T}\mcC^\alpha}\leq 1$,
where $u=\msS(\zeta,Z)$ and $A_\mfR:= \{(\zeta, Z\} \in \mcC^{\alpha_0}\times C_T\mcC^{\alpha+\delta}\,:\,\|\zeta\|_{\mcC^{\alpha_0}}+\|Z\|_{C_T\mcC^{\alpha+\delta}}<\mfR\}$.

We now fix for the remainder of the section $\zeta\in\mcC^{\alpha_0}_\mfm$, $T>0$, and $Z\in C_T\mcC^{\alpha+\delta}$ with $Z_0=0$.
We let $\eps_0>0$ sufficiently small such that $\msS^\eps(\zeta,Z)$ is well-defined and $\|u_\eps-u\|_{C_{\eta;T}\mcC^\alpha}\leq 1$ for all $\eps\in (0,\eps_0)$.
In the sequel we will always let $\eps\in(0,\eps_0)$.

We define $N_\varepsilon:=\{m\in\mbZ\,:\,|m|<\frac1\eps,\;m\neq 0 \}$
and further assume that $Z_{\eps;t}:=\hat\Pi_\varepsilon Z_{t}$ is of the form
\begin{equation}\label{eq:ApproxZ}
Z_{\varepsilon;t} =\sum_{m\in N_\eps} \int_0^t\ell_\eps(m) e^{-4\pi^2|m|^2(t-s)}\dd\hat{B}_{m;s} \,e_m,
\end{equation}
where $\{\hat B_m\}_{m\in\mbZ,m\neq 0}$
are functions $\hat B_m\in C([0,T],\mbC)$ which satisfy the reality condition
\begin{equation}\label{eq:reality}
\hat B_{-m} = \overline{\hat B_m}.
\end{equation}
We denote
\begin{equation*}
\hat{B}_\varepsilon := (\hat{B}_m)_{m\in N_\eps} \in \bar C_T\mbC^{N_\varepsilon},
\end{equation*}
where $\bar C_T\mbC^{N_\varepsilon}\subset C_T\mbC^{N_\eps}$ is the subspace of $(\hat B_m)_{m\in N_\eps}$ which satisfies the reality condition~\eqref{eq:reality}.

\begin{remark}\label{rem:RS_integrals}
We will henceforth identify $\bar C_T\mbC^{N_\varepsilon}$ with a subspace of $C_T C^{\infty}_0\subset C_T\mcC^{\alpha+\delta}_0$ by mapping $(\hat B_m)_{m\in N_\eps}$ to $Z_\eps$ via~\eqref{eq:ApproxZ}.
Observe that the integral in~\eqref{eq:ApproxZ} is well-defined as a Riemann--Stieltjes integral for any $\hat B \in \bar C_T\mbC^{N_\varepsilon}$.
For fixed $\zeta\in\mcC^{\alpha_0}$,
we will treat in this way
$u_\eps(\zeta,Z):=\msS^\eps(\zeta,Z)$ as a function of $\hat B_\eps$ whenever it is well-defined.
\end{remark}

We now prove the differentiability of $\msS^\varepsilon$ with respect to both arguments separately, using $\mcD$ for derivatives with respect to $\hat{B}_\varepsilon$ and $D$ for derivatives with respect to $\zeta$. For $\mfR\geq 1$, we recall the definition of $T_*(\mfR)$ given by \eqref{eq:1dT(R)Def}.
\begin{lemma}[Derivative in Noise]\label{lem:NoiseDerivative}
There exists an open neighbourhood $\mcO_{\hat{B}_\varepsilon} \subset \bar C_T\mbC^{N_\varepsilon}$ containing $\hat{B}_\varepsilon$ such that $u_\varepsilon(\zeta,\cdot)$ is Fr{\'e}chet differentiable as a mapping from $\mcO_{\hat{B}_\varepsilon}$ to $C_{T}\Pi_\eps L^2(\T)$. Furthermore, for any $f \in \bar C_T\mbC^{N_\varepsilon}$, such that $f\tzero =0$, the directional derivative $\mcD_f u_\varepsilon $ satisfies the equation
\begin{equation}\label{eq:DerivativeInNoise}
\begin{aligned}
\mcD_{f} u_{\varepsilon;t}&= \int_0^t e^{(t-s)\Delta}\partial_x \hat{\Pi}_\varepsilon
\left(2u_{\varepsilon;s}\mcD_f u_{\varepsilon;s}\partial_x \rho_{u_{\varepsilon;s}} 
+ u^2_{\varepsilon;s}\partial_x \rho_{\mcD_f u_{\varepsilon;s}} \right)\,\dd s \\
&\quad + \sum_{m\in N_\eps} \int_0^t \ell_\eps(m )e^{-4\pi^2|m|^2(t-s)}\dd f_{m;s}e_m.
\end{aligned}
\end{equation}
Finally, for $\eps\in (0,\eps_0)$, there exists a $C(T,\,\|\zeta\|_{\mcC^{\alpha_0}},\,\|Z_\eps\|_{C_T\mcC^{\alpha}})>0$ such that,
\begin{equation}\label{eq:NoiseDerivativeLocalBnd}
\|\mcD_f u_{\varepsilon}\|_{C_{T}\mcC^{\alpha}} \leq C \|f\|_{C_T\mbC^{N_\varepsilon}}.
\end{equation}
\end{lemma}
\begin{proof}
Integration by parts implies that,
for any $m\in\mbZ$ and $f\in C_T\C$ with $f_0=0$
\begin{equation*}
\int_0^t e^{-4\pi^2|m|^2(t-s)}\dd f_s = f_t - 4\pi^2|m|^2\int_0^t e^{-4\pi^2|m|^2(t-s)}f_s\,\dd s.
\end{equation*}
It follows that $Z_{\varepsilon}$ is a bounded, linear function of $\hat{B}_\varepsilon$ with values in  $ C_T \Pi_\varepsilon L^2_0(\mbT)$, and so is continuously Fr{\'e}chet differentiable. Furthermore, for any $f \in \bar C_T\C^{N_\eps}$ with $f_0=0$
\begin{equation*}
\mcD_f Z_{\varepsilon;t} = \sum_{m\in N_\eps} \left(f_{m;t} - 4\pi^2 |m|^2\int_0^t e^{-4\pi^2|m|^2(t-s)}f_{m;s}\,\dd s\right)\ell_\eps(m)e_m,
\end{equation*}
Regarding the approximate solution, $u_{\varepsilon;t}$, the mappings $h \mapsto h^2$ and $h\mapsto \partial_x \rho_h$ are Fr{\'e}chet differentiable on $\Pi_\varepsilon(L^2(\mbT))$, so the map
\begin{multline*}
\Phi_T\colon (z,(f_m)_{m\in N_\eps})\, \mapsto \,-z + e^{\,\cdot\,\Delta}\zeta_\eps + \int_0^{\,\cdot\,} e^{(\,\cdot\, -s)\Delta} \partial_x \hat{\Pi}_\varepsilon \left(z_{s}^2 \partial_x \rho_{z_{s}}\right)\,\dd s
\\
+\sum_{m\in N_\eps} \int_0^{\,\cdot\,}\ell_\eps(m) e^{-4\pi^2|m|^2(t-s)}\dd f_{m;s} \,e_m,
\end{multline*}
is Fr{\'e}chet differentiable as a mapping $\Phi_T\colon (C_{T}\Pi_\varepsilon L^2(\mbT),\bar C_T\mbC^{N_\varepsilon}) \to C_T\Pi_\varepsilon L^2(\mbT)$
and is such that $\Phi_T(u_\varepsilon, \hat{B}_\eps) =0$.
Moreover, writing $\msD$ for the Fr\'echet derivative,
\[
(\msD \Phi_{T_*})(u_\eps,\hat B_\eps)(\,\cdot\,,0)
\colon C_{T_*}\Pi_\varepsilon L^2(\mbT) \to C_{T_*}\Pi_\varepsilon L^2(\mbT)
\]
is a Banach space isomorphism
for $T_*(\|u_\eps\|_{C_T\mcC^\alpha})>0$ sufficiently small.
Applying the implicit function theorem for Banach spaces, \cite[Thm. 2.3]{ambrosetti_prodi_93}, we obtain that $u_\eps(\zeta,\cdot)|_{[0,T_*]}$ is Fr{\'e}chet differentiable in a neighbourhood of $\hat B_\eps$.
Patching together a sufficient (but finite) number of intervals of length $T_*$ to cover $[0,T]$, we obtain the first claim.

To derive \eqref{eq:DerivativeInNoise}, recall that the Fr{\'e}chet and Gateaux derivatives of a Fr{\'e}chet differentiable map agree.
Hence, for any $f\in\bar C_T\C^{N_\eps}$, $\mcD_f u_{\varepsilon} = \frac{\dd}{\dd \lambda} \msS^\eps(\zeta,\hat B_\eps + \lambda f_\eps)\big|_{\lambda =0}$
from which~\eqref{eq:DerivativeInNoise} follows now from the approximate equation,~\eqref{eq:1dApproximateEquation}.
We finally show~\eqref{eq:NoiseDerivativeLocalBnd}.
By the triangle inequality, the properties of $\hat{\Pi}_\varepsilon$, \eqref{eq:BesovGradElliptic} and applying \eqref{eq:HeatFlow}, for any $t\in [0,T]$,
\begin{equation*}
\|\mcD_f u_{\varepsilon;t}\|_{\mcC^{\alpha}} \lesssim \|u_{\varepsilon}\|^2_{C_{\eta;t}\mcC^{\alpha}} \int_0^t s^{-\frac{1}{2}-2\eta} \|\mcD_f u_{\varepsilon;s}\|_{\mcC^{\alpha}}\,\dd s + \sup_{s\in [0,t]}|f_s|_{\mbC^{N_\varepsilon}}.
\end{equation*}
Therefore, by Gr{\"o}nwall's inequality, there exists $C>0$ such that
\begin{equation*}
\|\mcD_f u_{\varepsilon;t}\|_{\mcC^{\alpha}} \lesssim \|f\|_{C_T \mbC^{N_\varepsilon}}
\exp
\big(
Ct^{\frac{1}2-2\eta}\|u_\varepsilon\|^2_{C_{\eta;t}\mcC^{\alpha}}
\big).
\end{equation*}
Due to the global existence of $u_{\varepsilon}\in C_{\eta;T}\mcC^{\alpha}$ (shown in Theorem \ref{th:1dApproxEquation} for $\eps\in (0,\eps_0)$ where $\eps_0$ depends on $\|\zeta\|_{\mcC^{\alpha_0}}, \|Z_\eps\|_{C_T\mcC^{\alpha}}$), for a new constant $C(T,\|\zeta\|_{\mcC^{\alpha_0}},\|Z\|_{C_T\mcC^{\alpha}})>0$,
\begin{equation*}
\|\mcD_f u_{\varepsilon}\|_{C_{T}\mcC^{\alpha}}
\leq C
\|f\|_{C_{ T}\mbC^{N_\varepsilon}}.
\end{equation*}
\end{proof}
Regarding the derivative of $u_\varepsilon(\zeta)$ with respect to the initial condition, for $g \in \mcC^{\alpha_0}$, we set $g_\varepsilon := \hat{\Pi}_\varepsilon g $ and then for any $0\leq s\leq t \leq T$ we let $J_{s,t}^\varepsilon g$ solve the equation
\begin{equation}\label{eq:1dApproxJacobian}
J^\varepsilon_{s,t}g = e^{(t-s)\Delta} g_\eps + \int_s^t e^{(t-r)\Delta}\partial_x \hat{\Pi}_\varepsilon[2u_{\varepsilon;r} \partial_x \rho_{u_{\varepsilon;r}}J_{s,r}^\varepsilon g
+
u^2_{\varepsilon;r}\partial_x \rho_{J_{s,r}^\varepsilon g}]\,\dd r.
\end{equation}
We show below that for any $\zeta \in \mcC^{\alpha_0}(\mbT)$ and $t\in[0,T]$, $J^\varepsilon_{0,t}g = D_{g}u_\varepsilon(\zeta)$, the directional derivative of $u_\varepsilon(\zeta)$ in $\zeta$.
Note that $J^\varepsilon$
satisfies the flow property, that is for $0\leq s\leq t \leq T$ one has $J^{\varepsilon}_{0,t} = J^\varepsilon_{s,t}J^\varepsilon_{0,s}$. In particular $J^\eps_{s,s}=\id$.
\begin{lemma}\label{lem:JacobianWellPosedBound}
There exists an open neighbourhood $\mcO_{\zeta_\varepsilon} \subset \Pi_\varepsilon(L^2(\mbT))$ containing $\zeta_\varepsilon$ such that $u_\varepsilon(\,\cdot\,,\hat B_\eps)$ is Fr{\'e}chet differentiable as a mapping from $\mcO_{\zeta_\varepsilon}$ to $C_{T}\Pi_\varepsilon(L^2(\mbT))$.
For any $g \in \mcC^{\alpha_0}$, the derivative is given by $D_{g} u_{\varepsilon;t}(\zeta) = J_{0,t}^\varepsilon g$.
Furthermore, setting $\mfR=\|Z\|_{C_T\mcC^{\alpha}}+\|\zeta\|_{\mcC^{\alpha_0}}$,
there exists $C(\mfR)>0$ and $T_*(\mfR)>0$
such that for all $t \in (0,T_*]$
\begin{equation}\label{eq:ApproxJacobianCAlphaBound}
\sup_{s \in [0,t]} s^\eta\|J_{0,s}^\varepsilon g\|_{\mcC^{\alpha}} \leq C \|g\|_{\mcC^{\alpha_0}}.
\end{equation}
\end{lemma}
\begin{proof}
The same argument as in the proof of Lemma \ref{lem:NoiseDerivative} shows that the map $\Pi_\varepsilon L^2_\mfm(\mbT)\ni \zeta_\varepsilon\mapsto u_{\varepsilon}(\zeta)\in C_T\Pi_\varepsilon(L^2_\mfm(\mbT))$ is Fr{\'e}chet differentiable in a neighbourhood of $\zeta_\varepsilon$. It is then readily checked that on $\mcO_{\zeta_\varepsilon}$, for any $g \in \mcC^{\alpha_0}$ the Fr{\'e}chet derivative is equal to the map $t\mapsto J^\varepsilon_{0,t}g$.

To prove \eqref{eq:ApproxJacobianCAlphaBound}, observe that
\begin{equation*}
\|	J^\varepsilon_{0,s}g \|_{\mcC^{\alpha}} \lesssim s^{-\frac{\alpha-\alpha_0}{2}}\|g\|_{\mcC^{\alpha_0}} + \|u_{\varepsilon}\|^2_{C_{\eta;s}\mcC^{\alpha}}\int_0^s (s-r)^{-\frac{1}{2}}r^{-2\eta}\|J^\varepsilon_{0,r}g\|_{\mcC^{\alpha}}\,\dd r. 
\end{equation*}
Therefore, for any $t\in (0,T]$,
\begin{equation*}
\sup_{s\in [0,t]}s^\eta \|	J^\varepsilon_{0,s}g \|_{\mcC^{\alpha}} \lesssim t^{\eta-\frac{\alpha-\alpha_0}{2}}\|g\|_{\mcC^{\alpha_0}} + t^{\frac{1}{2}-\eta}\|u_{\varepsilon}\|^2_{C_{\eta;t}\mcC^{\alpha}}\sup_{s\in [0,t]}s^\eta \|	J^\varepsilon_{0,s}g \|_{\mcC^{\alpha}} .
\end{equation*}
Since
$\|u_\eps-u\|_{C_{\eta;T}\mcC^\alpha}\leq 1$,
by Theorem \ref{th:1dRemainderLWP}
there exists a $T_*(\mfR)\in (0,1)$ such that $\|u_{\varepsilon}\|_{C_{\eta;T_*}\mcC^{\alpha}}\leq 2$.
Hence, for all $t\in (0,T_*]$,
\begin{equation*}
\sup_{s \in [0,t]}s^\eta \|	J^\varepsilon_{0,s}g \|_{\mcC^{\alpha}} \lesssim t^{\eta-\frac{\alpha-\alpha_0}{2}}\|g\|_{\mcC^{\alpha_0}} + t^{\frac{1}{2}-\eta}\sup_{s \in [0,t]}s^\eta \|	J^\varepsilon_{0,s}g \|_{\mcC^{\alpha}},
\end{equation*}
so then choosing a sufficiently small time $t_1(\mfR) \in (0,t]$,
\begin{equation*}
\sup_{s \in [0,t_1]}s^\eta \|	J^\varepsilon_{0,s}g \|_{\mcC^{\alpha}} \lesssim t_1^{\eta-\frac{\alpha-\alpha_0}{2}}\|g\|_{\mcC^{\alpha_0}}.
\end{equation*}
Repeating this procedure, we find a constant $C:=C(\mfR)>0$ such that 
\begin{equation*}
\sup_{s \in [0,t]}s^\eta \|	J^\varepsilon_{0,s}g \|_{\mcC^{\alpha}} \leq C \|g\|_{\mcC^{\alpha_0}},\quad \forall\, t \in (0,T_*].
\end{equation*}
\end{proof}
We finally reintroduce probability and consider in the remainder of the section
a finite dimensional approximation $B_{\varepsilon;t}$ of the white noise
defined by
\begin{equation*}
B_{\varepsilon;t} := \sum_{m\in N_\eps} e_m\hat{B}_{t;m},\quad \hat B_{m;t} := \xi (\indic_{[0,t]}\times e_m), \quad \text{ for }m \in N_\eps.
\end{equation*}
Note that $(\hat B_m)_{m\in N_\eps}$ is a family of complex valued Brownian motions which satisfy the reality condition~\eqref{eq:reality}.
Our approximation of the SHE corresponding to~\eqref{eq:ApproxZ} is then
\begin{equation}\label{eq:ApproxStatSHE}
v_{\varepsilon;s,t}:=\hat{\Pi}_\varepsilon v_{s,t} =\sum_{m\in N_\eps} \int_s^t \ell_\eps(m) e^{-4\pi^2|m|^2(t-r)}\dd\hat B_{m;r} \,e_m,
\end{equation}
By Remark~\ref{rem:RS_integrals}, since $\hat B_\eps\in \bar C_T\C^{N_\eps}$,  $(s,t)\mapsto v_{\varepsilon;s,t}$ is well defined pathwise.
Furthermore, by Property~\ref{ass:ProjectionProperty2}, there exists a $\mbP$-null set $\mcN \subset \Omega$ such that, fixing any realisation $\xi(\omega)$ for $\omega \in \Omega\setminus\mcN$ gives realisations of $v_{\varepsilon}(\omega),\,v(\omega)$ and for which $v_{\varepsilon;0,\,\cdot\,}(\omega)\rightarrow v_{0,\,\cdot\,}(\omega)$ in $\mcC^\kappa_T\mcC^{\alpha-2\kappa}$ for every $\kappa \in [0,1/2)$.
In the rest of the subsection, we will let $v_\eps$ denote the random path $t\mapsto v_{\eps;0,t}$.

For each $\varepsilon>0$,
the Cameron--Martin space of $\hat B_\varepsilon$ is
\begin{multline*}
\mcC\mcM^\varepsilon = \bar W^{1,2}_0\left([0,T];\mbC^{N_\varepsilon}\right)
:= \bigg\{ f \in L^2([0,T];\,\mbC^{N_\varepsilon})\,:\,
\partial_t f \in L^2\left([0,T];\,\mbC^{N_\varepsilon}\right),
\,
\\
f_{-m} = \bar f_m,\,
f\tzero=0 \bigg\}.
\end{multline*}
By the Sobolev embedding, $W^{1,2}(\mbR)\hookrightarrow \mcC^{1/2}(\mbR)$, we see $\mcC\mcM^\varepsilon\subset \bar C_T \mbC^{N_\varepsilon}$. Therefore, Lemma~\ref{lem:NoiseDerivative} applies with $f \in \mcC\mcM^\varepsilon$. We also choose a smooth, compactly supported, cut-off function $\Theta :\mbR_+\rightarrow [0,1]$ such that $\Theta(z) = 1$ for $z<\frac{1}{2}$ and $\Theta(z)=0$ for $z\geq 1$.

We introduce the notion of right sided derivatives, $\mcD^+$, of $\|Z\|_{C_T\mcC^\alpha}$, which, for $f \in \bar C_T\mcC^{\alpha}$, is defined by
\begin{equation*}
\mcD^+_{f} \|Z\|_{C_T\mcC^{\alpha}}:= \lim_{\lambda \searrow 0} \frac{\|Z+\lambda f\|_{C_T\mcC^{\alpha}} -\|Z\|_{C_T\mcC^{\alpha}}}{\lambda}.
\end{equation*}
\begin{theorem}[Bismut--Elworthy--Li Formula]\label{th:BELFormula}
Let $t>0$, $\zeta \in \mcC^{\alpha_0}_\mfm(\mbT)$, $\Phi \in \mcC^1_b(\mcC^{\alpha_0}(\mbT))$ and $f^\varepsilon$ be a $\mcC\mcM^\varepsilon$ valued process taking for which $\partial_t f^\varepsilon$ is adapted and such that $\|\partial_t f^\eps\|_{L^2([0,t]; \mbC^{N_\varepsilon})}\leq C$, $\mbP$-a.s. for a deterministic constant $C:=C(t)$.
Then we have the identity
\begin{equation}\label{eq:BELFormula}
\begin{aligned}
\hspace{-0.4em}	\mbE\big[ \mcD_{f^\varepsilon} [\Phi(u_{\varepsilon;t}(\zeta,\hat B_\eps))]\Theta(\|v_{\varepsilon}\|_{C_t\mcC^{\alpha}})\big]
&=
\mbE\Big[
\Phi(u_{\varepsilon;t}(\zeta,\hat B_\eps))
\Theta(\|v_{\varepsilon}\|_{C_t\mcC^{\alpha}})\int_0^t \partial_t f^\varepsilon_s \cdot\,\dd \hat{B}_{\varepsilon;s} \Big]\\
&\quad -
\mbE\Big[\Phi(u_{\varepsilon;t}(\zeta,\hat B_\eps))\Theta'(\|v_{\varepsilon}\|_{C_t\mcC^{\alpha}})\mcD^+_{f^\eps} \|v_{\varepsilon}\|_{C_t\mcC^{\alpha}}
\Big],
\end{aligned}
\end{equation}
where we recall that $f^\eps$ is identified with a function in $C_t C^\infty_0$ as in Remark~\ref{rem:RS_integrals}.
\end{theorem}
\begin{proof}
For any $f^\varepsilon \in \mcC\mcM^\varepsilon$ and $\lambda >0$, define the shift operator $\mcT^\lambda_{f^\varepsilon} \hat{B}_\varepsilon := \hat{B}_\varepsilon+\lambda f^\varepsilon$.
Define further
\begin{align*}
	u^{\lambda;f}_{\varepsilon}(\zeta)& := \msS^\varepsilon\left(\zeta,	\mcT^\lambda_{f^\varepsilon} \hat{B}^\varepsilon\right),\\
	\mcT^\lambda_{f^\varepsilon} v_{\varepsilon;0,t} &:=\sum_{m\in N_\eps} \int_0^t \ell_\eps(m) e^{-4\pi^2|m|^2(t-r)}(\dd\hat{B}^m_r +\lambda\partial_t f^\varepsilon_{m;r} \,\dd r)\,e_m.
\end{align*}
We now look for a family of measures $\mbP_\lambda$ such that the law of $\mcT^\lambda_{f^\varepsilon}\hat{B}^\varepsilon$ under $\mbP_\lambda$ is equal to the law of $\hat{B}^\varepsilon$ under $\mbP$ for each $\lambda$.
The assumptions on $f^\eps$ imply that Novikov's condition~\cite[Ch. 8, Prop. 1.15]{revuz_yor_08} is satisfied by $\lambda \partial_t f^\eps$.
Hence, by Girsanov's theorem~\cite[Ch. 4, Thm. 1.4]{revuz_yor_08}, the desired measure $\mbP_\lambda$ is given by
\begin{equation*}
    	\frac{\dd\mbP_\lambda}{\dd \mbP} := A^\lambda_t, \quad 	A_t^\lambda := \exp \left( -\lambda \int_0^t \partial_t f^\varepsilon_s \,\cdot \dd \hat{B}^\varepsilon_s -\frac{\lambda^2}{2}\int_0^t |\partial_tf^\varepsilon_s|^2\,\dd s\right).
\end{equation*}
As a consequence, and using $\mbE_{\mbP_\lambda}\left[\,\cdot\,\right] = \mbE \left[ \,\cdot\, A^\lambda_t\right]$, we have
\begin{equation*}
	\frac{\dd}{\dd\lambda}\mbE\left[ \Phi(u^{\lambda;f^\varepsilon}_{\varepsilon;t}(\zeta))\Theta(\|\mcT^\lambda_{f^\varepsilon} v_{\varepsilon}\|_{C_t\mcC^{\alpha}})A^{\lambda }_t \right] \bigg|_{\lambda=0+}= 0.
\end{equation*}
We now show that we may exchange differentiation and integration in the above expression, after which we will derive the BEL formula, \eqref{eq:BELFormula}.
%
%
Let us define the difference operator $\Delta_\lambda F := F(\lambda)-F(0)$ so that, for example,
\begin{equation*}
    \Delta_\lambda \Phi(u^{\,\cdot\,;f^\varepsilon}_{\eps;t}(\zeta)) = \Phi(u^{\lambda;f^\varepsilon}_{\eps;t}(\zeta))- \Phi(u_{\eps;t}(\zeta)).
\end{equation*}
We demonstrate that the family $\{\frac{1}{\lambda}(\Delta_{\lambda}(\Phi(u^{\,\cdot\,;f^\varepsilon}_{\varepsilon;t}(x))\Theta(\|\mcT^{\,\cdot\,}_{f^\varepsilon} v_{\varepsilon;0,\,\cdot\,}\|_{C_T\mcC^{\alpha}})A^{\,\cdot\,}_t))\}_{\lambda \in (0,1]}$ is $\mbP$-uniformly integrable.
For any $\lambda\in (0,1]$ we observe that,
\begin{align*}
    \Delta_{\lambda}(\Phi(u^{\,\cdot\,;f^\varepsilon}_{\varepsilon;t}(\zeta))\Theta(\|\mcT^{\,\cdot\,}_{f^\varepsilon} v_{\varepsilon}\|_{C_t\mcC^{\alpha}})A^{\,\cdot\,}_t) &= \Delta_{\lambda}\Phi(u^{\,\cdot\,;f^\varepsilon}_{\varepsilon;t}(\zeta))\Theta(\|\mcT^{\lambda}_{f^\varepsilon} v_{\varepsilon}\|_{C_t\mcC^{\alpha}})A^{\lambda}_t\\
     &\quad +\Phi(u_{\varepsilon;t}(\zeta))\Delta_{\lambda}\Theta(\|\mcT^{\,\cdot\,}_{f^\varepsilon} v_{\varepsilon}\|_{C_t\mcC^{\alpha}})A^{\lambda}_t \\
    &\quad+\Phi(u_{\varepsilon;t}(\zeta))\Theta(\|v_{\varepsilon}\|_{C_t\mcC^{\alpha}})\Delta_{\lambda} A^{\,\cdot\,}_t
\end{align*}
For the first term, by the assumption that $\Phi\in \mcC^1_b(\mcC^{\alpha_0}(\mbT))$,
\begin{align*}
|\Delta_{\lambda}(\Phi(u^{\,\cdot\,;f^\varepsilon}_{\varepsilon;t}(\zeta))|&\leq \|\Phi\|_{\mcC^{1}_b} \|u^{\lambda;f^\varepsilon}_{\varepsilon;t}(\zeta) - u_{\varepsilon;t}(\zeta)\|_{\mcC^{\alpha_0}}\\
& \leq \|\Phi\|_{\mcC^{1}_b} \int_0^\lambda \left\|\mcD_{f}u^{\tilde{\lambda};f}_{\varepsilon;t}(\zeta)\right\|_{\mcC^{\alpha}}\dd \tilde{\lambda}.
\end{align*}
If $\|\mcT^{{\lambda}}_{f^\eps}v_\eps\|_{C_t\mcC^{\alpha}} > 1$, then the first term vanishes,
so suppose that $\|\mcT^{{\lambda}}_{f^\eps}v_\eps\|_{C_t\mcC^{\alpha}}\leq 1$.
Hence, due to the bound on $f^\eps$, there exists a deterministic $C>0$ such that $\|\mcT^{\tilde{\lambda}}_{f^\eps}v_\eps\|_{C_t\mcC^{\alpha}}\leq C$ for all $\tilde\lambda\in[0,\lambda]$.
Therefore, applying \eqref{eq:NoiseDerivativeLocalBnd} uniformly for $\tilde{\lambda}\in (0,\lambda]$, we have the bound
$$
|\Delta_{\lambda}\Phi(u^{\,\cdot\,;f^\varepsilon}_{\varepsilon;t}(\zeta))\Theta(\|\mcT^{\lambda}_{f^\varepsilon} v_{\varepsilon}\|_{C_t\mcC^{\alpha}})A^{\lambda}_t| \,\lesssim_{f,\Theta,t} \lambda \|\Phi\|_{\mcC^1_b} A^\lambda_t,
$$
where the implicit constant is independent of the realisation of $v$ and $\lambda \in (0,1]$. For the second term, arguing similarly and using our assumptions on $\Theta$ we obtain the similar bound,
$$
|\Phi(u_{\varepsilon;t}(\zeta))\Delta_{\lambda}\Theta(\|\mcT^{\,\cdot\,}_{f^\varepsilon} v_{\varepsilon}\|_{C_T\mcC^{\alpha}})A^{\lambda}_t|\,\lesssim_{f,\Theta,t} \lambda \|\Phi\|_{\infty} A^\lambda_t.
$$
Finally, for the third term,
$$
|\Phi(u_{\varepsilon;t}(\zeta))\Theta(\| v_{\varepsilon}\|_{C_t\mcC^{\alpha}})\Delta_{\lambda} A^{\,\cdot\,}_t| \lesssim_{\Theta} \|\Phi\|_{\infty} \int_0^\lambda \left|\frac{\dd}{\dd \lambda} A^{\tilde{\lambda}}_t\right|\dd \tilde{\lambda}.
$$
Therefore there exists a constant $C(f,\Theta,t,\|\Phi\|_{\mcC^1_b})>0$ such that for all $\lambda \in (0,1]$ and $p\geq 1$
$$
\mbE\left[\left|\frac{ \Delta_{\lambda}(\Phi(u^{\,\cdot\,;f^\varepsilon}_{\varepsilon;t}(\zeta))\Theta(\|\mcT^{\,\cdot\,}_{f^\varepsilon} v_{\varepsilon}\|_{C_t\mcC^{\alpha}})A^{\,\cdot\,}_t)}{\lambda}\right|^p\right] \leq C \left(\sup_{\lambda \in (0,1]}\mbE[(A^\lambda_t)^p]+ \sup_{\lambda\in (0,1]}\mbE\left[\left|\frac{\dd}{\dd \lambda} A^{\lambda}_t\right|^p\right]\right),
$$
where we applied Jensen's inequality to obtain $(\int_0^\lambda |\frac{\dd}{\dd \lambda} A^{\tilde{\lambda}}_t|\dd \tilde{\lambda})^p \leq
\lambda^{p-1}\int_0^\lambda |\frac{\dd}{\dd \lambda} A^{\tilde{\lambda}}_t|^p\dd \tilde{\lambda}$.
For the first term we have $(A^\lambda_t)^p = A^{p\lambda}_t \exp\left(\frac{p^2-p}{2}\lambda^2\int_0^t |\partial_t f^\eps_s|^2\dd s\right)$ where $A^{p\lambda}$ is also a strictly positive martingale with expectation $1$ and $\exp\left(\frac{p^2-p}{2}\lambda^2\int_0^t |\partial_t f^\eps_s|^2\dd s\right)$ is uniformly bounded in $L^\infty(\Omega;\mbR)$ across $\lambda \in (0,1]$ due to the assumptions on $f^\eps$. Regarding the second term, using the explicit form,
\begin{equation}\label{eq:ADiVIdentity}
  \frac{\dd}{\dd \lambda} A^{\lambda}_t= -A^{\lambda}_t \left(\int_0^t \partial_t f^\eps_s\cdot \dd \hat B_{\eps;s} + \lambda \int_0^t|\partial_t f^\eps_s|^2\dd s \right)
\end{equation}
and Cauchy--Schwarz, we see that
\begin{align*}
    \mbE\left[\left|\frac{\dd}{\dd \lambda} A^{\tilde{\lambda}}_t\right|^p\right]^2 \leq \mbE\left[(A^{\tilde{\lambda}}_t)^{2p}\right] \mbE\left[\left(\int_0^t \partial_t f^\eps_s\cdot \dd \hat B_s^\eps + \tilde{\lambda} \int_0^t|\partial_t f^\eps_s|^2\dd s \right)^{2p}\right].
\end{align*}
The first factor is uniformly bounded for all $\tilde{\lambda} \in (0,1]$ by the same arguments as we applied for $(A^\lambda_t)^p$. The second factor can be controlled using the Burkholder--Davis--Gundy inequality and again our assumption of uniform boundedness on $\|\partial_tf^\eps\|_{L^2([0,t])}$. Hence
$$
\sup_{\lambda \in (0,1]}\mbE\left[\,\left|\frac{ \Delta_{\lambda}(\Phi(u^{\,\cdot\,;f^\varepsilon}_{\varepsilon;t}(x))\Theta(\|\mcT^{\,\cdot\,}_{f^\varepsilon} v_{\varepsilon}\|_{C_t\mcC^{\alpha}})A^{\,\cdot\,}_t)}{\lambda}\right|^p\,\right] <\infty,
$$
and we have uniform integrability.
Therefore, by Vitali's convergence theorem, we may exchange differentiation and expectation to give that 
$$
\mbE\left[\frac{\dd}{\dd\lambda} \left(\Phi(u^{\lambda;f^\varepsilon}_{\varepsilon;t}(\zeta))\Theta(\|\mcT^\lambda_{f^\varepsilon} v_{\varepsilon}\|_{C_t\mcC^{\alpha}})A^{\lambda }_t\right)\bigg|_{\lambda=0+} \right] = 0.
$$
In order to conclude the proof it only remains to show that the identities
\begin{align*}
\mcD^+_{f^\varepsilon}\Phi\left(u_{\varepsilon;t}(\zeta)\right) &= \mcD_{f^\varepsilon}\Phi\left(u_{\varepsilon;t}(\zeta)\right),\\
\frac{\dd}{\dd\lambda}A^{\lambda }_t \bigg|_{\lambda=0+} &= -\int_0^t \partial_t f^\eps_s\cdot \dd\hat B_{\eps;s},\\
\frac{\dd}{\dd\lambda}\Theta\left(\|\mcT^\lambda_{f^\varepsilon} v_{\varepsilon}\|_{C_t\mcC^{\alpha}}\right)\bigg|_{\lambda=0+}A^{0 }_t &=\Theta'(\|v_{\varepsilon}\|_{C_t\mcC^{\alpha}})\mcD^+_{f^\varepsilon} \|v_{\varepsilon}\|_{C_t\mcC^{\alpha}}.
\end{align*}
The first follows directly from Lemma \ref{lem:NoiseDerivative}, since one sided derivatives and full derivatives agree for any Fr\'echet differentiable map. The second identity follows after setting $\lambda=0$ in \eqref{eq:ADiVIdentity} and observing that $A^0_t\equiv1$. The final identity follows in the same way from our assumptions on $\Theta$ and the chain rule.
\end{proof}
We can apply the identity \eqref{eq:BELFormula} to obtain an initial bound on the difference of the semi-group acting on elements of $\mcC^1_b(\mcC^{\alpha_0})$, from two initial points.
The following result and its proof is close to that
of~\cite[Prop.~5.8]{tsatsoulis_weber_18}.
\begin{proposition}\label{prop:SemiGroupPreLocalLip}
Let $\zeta,\,\tilde{\zeta} \in \mcC_\mfm^{\alpha_0}$ with $\tilde{\zeta} \in B_1(\zeta)$, set $\mfR:= 1+2\|\zeta\|_{\mcC^{\alpha_0}}$
and let $T_*(\mfR)>0$ be defined according to \eqref{eq:1dT(R)Def}. Then there exists a constant $C:=C(\mfR,\alpha,\alpha_0,\eta)>0$ and exponent $\theta:=\theta(\eta)>0$ such that for any $\Phi \in \mcC^1_{b}(\mcC^{\alpha_0}_\mfm)$
and $t \in (0,T_*]$,
\begin{equation}\label{eq:PointWiseSemiGroupBound}
|P_t\Phi(\zeta)-P_t\Phi(\tilde{\zeta})| \leq C\frac{1}{t^\theta} \|\Phi\|_{L^\infty} \|\zeta-\tilde{\zeta}\|_{\mcC^{\alpha_0}} + 2\|\Phi\|_{L^\infty} \mbP(\|v\|_{C_t\mcC^{\alpha}}>1).
\end{equation}
\end{proposition}
\begin{proof}
First, defining the semi-group for the approximate equation and applying the triangle inequality, for every $\Phi\in\mcC^1_b(\mcC^{\alpha_0})$
one obtains
\begin{align}
	|P^\varepsilon_t\Phi(\zeta)-P^\varepsilon_t\Phi(\tilde{\zeta})| &\leq \left|\mbE\left[ (\Phi(u_{\varepsilon;t}(\zeta)) - \Phi(u_{\varepsilon;t}(\tilde{\zeta})))\Theta(\|v_{\varepsilon}\|_{C_t\mcC^{\alpha} } ) \right]\right|\notag \\
	&\quad + \left|\mbE\left[ (\Phi(u_{\varepsilon;t}(\zeta))-\Phi(u_{\varepsilon;t}(\zeta)))\left(1-\Theta(\|v_{\varepsilon}\|_{C_t\mcC^{\alpha}})\right) \right]\right|\notag \\
	&=: I_1(t) + I_2(t).
\end{align}
For the second term,
\begin{equation}\label{eq:I2Bound}
	I_2(t)   \leq 2 \|\Phi\|_{L^\infty}\mbP\left( \|v_{\varepsilon}\|_{C_t\mcC^{\alpha}}>1 \right),
\end{equation}
and using the fact that $v_\eps \rightarrow v$ in $C_t\mcC^{\alpha}$ in probability
we obtain the second term on the right hand side of \eqref{eq:PointWiseSemiGroupBound}. Regarding the first term,
we use the fundamental theorem of calculus along with Fubini to write,
\begin{equation*}
	I_1(t)= \left|\int_0^1 \mbE \left[ D_{\zeta-\tilde{\zeta}} [\Phi(u_t(\zeta+\lambda(\zeta-\tilde{\zeta}))] \Theta(\|v_{\varepsilon}\|_{C_t\mcC^{\alpha}} ) \right]\dd\lambda\,\right|.
\end{equation*}
We now note that for any $f^\varepsilon \in \mcC\mcM^\varepsilon$, it follows from the mild forms of $\mcD_{f^\eps}u_{\eps;t}$ and $J^\eps_{s,t}\partial_t f^\eps_{m;s}$ given in  \eqref{eq:DerivativeInNoise} and \eqref{eq:1dApproxJacobian} respectively, that,
\begin{equation*}
\mcD_{f^\eps} u_{\varepsilon;t}(\zeta) = \sum_{m\in N_\eps}\int_0^t J^\varepsilon_{s,t}\partial_t f^\eps_{m;s} e_m \,\dd s.
\end{equation*}
Denote $z_\lambda:= \zeta-\lambda(\zeta-\tilde{\zeta})\in \mcC^{\alpha_0}$ and let $f^\varepsilon \in \mcC\mcM^\varepsilon$ such that $\partial_t f^{\eps}_{m;s} = \frac{1}{t}\langle J^\eps_{0,s}(\zeta-\tilde{\zeta}),e_m\rangle$ for every $m\in N_\eps$, where $J^{\varepsilon}_{0,s}(\zeta-\tilde{\zeta}) = D_{\zeta-\tilde{\zeta}} u_{\varepsilon;s}(z_\lambda)$ and for notational ease, we have suppressed the dependence on $z_\lambda$ in $J^{\varepsilon}_{0,s}(\zeta-\tilde{\zeta})$.
Then
\begin{equation}\label{eq:NoiseVsInitialDerivativeIdentity}
\mcD_{f^\varepsilon} u_{\varepsilon;t}(z_\lambda) = J_{0,t}^\varepsilon (\zeta-\tilde{\zeta}) = D_{\zeta-\tilde\zeta}u_{\varepsilon;t}(z_\lambda).
\end{equation}
Furthermore, for $f^\eps$ defined in this way, $\partial_t f^\eps\in L^\infty(\Omega;L^2([0,t];\mbC^{N_\eps}))$.
Note also that $\|z_\lambda\|_{\mcC^{\alpha_0}}\leq \mfR$ for all $\lambda \in (0,1)$, so the local bounds of Lemmas \ref{lem:NoiseDerivative} and \ref{lem:JacobianWellPosedBound} both hold uniformly in $z_\lambda$. Using \eqref{eq:NoiseVsInitialDerivativeIdentity} and \eqref{eq:BELFormula} we obtain the bound
\begin{multline*}
	\left|\mbE \left[ D_{\zeta-\tilde\zeta}[ \Phi(u_{\varepsilon;t}(z_\lambda))] \Theta(\|v_{\varepsilon}\|_{C_t\mcC^{\alpha}}) \right]\right|
	\\
	\leq \frac{1}{t}\|\Phi\|_{L^\infty} \Big(  \mbE\Big[ \,\Big| \int_0^t\langle J^{\varepsilon}_{0,s}(\zeta-\tilde{\zeta}),\dd W_{\varepsilon;s}\rangle \,\Theta(\|v_{\varepsilon}\|_{C_t\mcC^{\alpha} } ) \Big|\,\,\Big]
\\
	+	\mbE\left[\,\left|\, \Theta'(\|v_{\varepsilon}\|_{C_t\mcC^{\alpha} } )\mcD_{ f^\varepsilon}^+\|v_{\eps}\|_{C_t\mcC^{\alpha}} \,\right|\,\right] \Big).
\end{multline*}
For the first term, one can apply Cauchy--Schwarz, It{\^o}'s isometry and the results of Lemma \ref{lem:JacobianWellPosedBound}
to obtain
\begin{align}
	\mbE\left[ \,\left| \int_0^t\langle J^{\varepsilon}_{0,s}(\zeta-\tilde{\zeta}),\dd W_{\varepsilon;s}\rangle \,\Theta(\|v_{\varepsilon;0,\,\cdot\,}\|_{C_t\mcC^{\alpha} } ) \right|\right]
	&\leq
	\mbE\left[
	\int_0^t \| J^{\varepsilon}_{0,s}(\zeta-\tilde{\zeta})\|^2_{L^2}\dd s
	\right]^{1/2}\notag
	\\
	&\lesssim_{\mfR,\alpha,\alpha_0,\eta}  t^{\frac{1}{2}-\eta}\|\zeta-\tilde{\zeta}\|_{\mcC^{\alpha_0}}. \label{eq:I11Bound}
\end{align}
%
For the second term, using the explicit form of the directional derivative and Lemma \ref{lem:JacobianWellPosedBound} again applied to our choice of $f^\eps$, we have
\begin{align}
	\mbE\left[\,\left|\, \Theta'(\|v_{\varepsilon}\|_{C_t\mcC^{\alpha} } )\mcD_{ f^\varepsilon}^+\|v_{\eps}\|_{C_t\mcC^{\alpha}} \,\right|\,\right] &\leq t^{-\eta}\mbE\left[\sup_{s \in [0,t]}s^\eta\|J^\varepsilon_{0,s}(\zeta-\tilde{\zeta})\|_{\mcC^{\alpha}}|\Theta'(\|v_{\eps}\|_{C_t\mcC^{\alpha}})| \right] \notag\\
	&\leq C t^{-\eta}\|\zeta-\tilde{\zeta}\|_{\mcC^{\alpha_0}},\label{eq:I12Bound}
\end{align}
 where in the last line we also used that $\Theta'$ is bounded.

Combining \eqref{eq:I2Bound}, \eqref{eq:I11Bound} and \eqref{eq:I12Bound} yields the stated inequality,~\eqref{eq:PointWiseSemiGroupBound}, for the approximate semi-group $P^\eps$. Taking the limit $\eps \rightarrow 0$ and applying the dominated convergence theorem concludes the proof.
\end{proof}
We conclude this section by obtaining a local H{\"o}lder bound for the dual semi-group.
It follows from this bound that,  for all $t>0$, $P_t \colon \mcB_b(\mcC_\mfm^{\alpha_0})\rightarrow \mcC_b(\mcC_\mfm^{\alpha_0})$, which concludes our proof of the strong Feller property for $(P_t)_{t>0}$.
\begin{theorem}\label{th:SemiGroupLocalLipschitz}
Let $\zeta,\,\tilde{\zeta} \in \mcC_\mfm^{\alpha_0}$ with $\tilde{\zeta} \in B_1(\zeta)$. Then there exists $C>0$, $\theta \in (0,1)$ and $\sigma>0$ such that, for every $t \geq 1$,
\begin{equation*}
\|P^*_{t}\delta_\zeta - P^*_{t}\delta_{\tilde{\zeta}}\|_{\TV} \,\,\leq C(1+\|\zeta\|_{\mcC^{\alpha_0}})^\sigma\|\zeta-\tilde{\zeta}\|^\theta_{\mcC^{\alpha_0}}.
\end{equation*} 
\end{theorem}
\begin{proof}[Sketch of Proof]
It follows from \cite[Lem.~7.1.5]{daprato_zabczyk_96} that \eqref{eq:PointWiseSemiGroupBound} is equivalent to the statement,
\begin{equation*}
    \|P^*_{t}\delta_\zeta - P^*_{t}\delta_{\tilde{\zeta}}\|_{\TV} \,\,\leq C\frac{1}{t^\theta}\|\zeta-\tilde{\zeta}\|_{\alpha_0} + 2\mbP(\|v\|_{C_t\mcC^{\alpha}}>1).
\end{equation*}
From this bound, one may follow exactly the steps of \cite[Thm.~5.10]{tsatsoulis_weber_18}. The key idea is to use the Gaussian nature of $v$ to obtain the control $\mbP(\|v\|_{C_t\mcC^{\alpha}}> 1) \lesssim t^{\theta_2}$ for some $\theta_2\in (0,1)$.
Then we use the explicit form of $T_\ast(\mfR)$ defined in \eqref{eq:1dT(R)Def} and monotonicity of the map $t\mapsto \|P^*_{t}\delta_\zeta - P^*_{t}\delta_{\tilde{\zeta}}\|_{\TV}$ to obtain the final result.
\end{proof}
\subsection{Full Support}\label{subsec:FullSupport}
We demonstrate that $u_T(\zeta)$, and thus any invariant measure $\nu_\zeta$ for $(P_t)_{t>0}$ as in Theorem~\ref{thm:existence_invar_measures},
has full support in $\mcC^{1/2-\delta}_{\bar{\zeta}}(\mbT)$ for any $\delta \in (0,1/2)$. In this subsection we are not concerned with the behaviour of the solution near zero, so until the start of Section \ref{subsec:ExponentialErgodicity} we just consider $\alpha \in (0,1/2)$ and $\mfm\in\mbR$. 

Let $L^2_0(\mbR_+\times \mbT)$ be the space of square integrable functions on $\mbR_+\times \mbT$ such that for any $t\geq 0$, $\bar{f}_t=0$. Then for any $T\geq 0$ we define the Cameron--Martin space of $v :=v_{0,\,\cdot\,}$,
\begin{equation*}
\msH_T := \left\{ h \colon [0,T] \times \mbT \to \mbR \,:\, h_t =\int_0^t e^{(t-s)\Delta}f_s\,\dd s,\, f \in L^2_0(\mbR_+ \times \mbT) \right\} ,
\end{equation*}
Note that by Theorems~\ref{th:BesovEmbedding} and~\ref{th:BesovHeatFlow}, $\msH_T$ is continuously and densely embedded in $\{h\in C_T\mcC^{\alpha}_0\,:\,h(0)=0\}$. The following is a direct consequence of the Cameron--Martin theorem.
\begin{lemma}[Theorem~3.6.1 of \cite{bogachev_1998}]\label{lem:SHESupport}
Let $T>0$ and $\mcL(v)= (v)\# \mbP \in \mcP(C_T\mcC^\alpha)$ be the law of $v$. Then $\supp (\mcL(v)) =\overline{\msH_T}^{\|\,\cdot\,\|_{C_T\mcC^{\alpha}}}$.
\end{lemma}
In the following theorem, we treat $\mcL(u_T(\zeta))$ as a probability measure on $\mcC^\alpha_\mfm(\mbT)$.
\begin{theorem}\label{th:SolFullSupport}
Let $T>0$, $\zeta \in \mcC^{\alpha}_{\mfm}(\mbT)$. Then
\begin{equation*}
\supp(\mcL(u_T(\zeta)))= \mcC^{\alpha}_{\mfm}(\mbT).
\end{equation*}
\end{theorem}
\begin{proof}
We first show 
\begin{equation}\label{eq:SolMapClosureInclusion}
\overline{\left\{ \msS_T(\zeta,h)\,:\,h\in \msH_T \right \}}^{\|\,\cdot\,\|_{\mcC^{\alpha}}}\subseteq \supp(\mcL(u_T(\zeta))) .
\end{equation}
Recall that the map $\msS_T(\zeta,\,\cdot\,)\colon C_T\mcC^{\alpha}_0(\mbT)\rightarrow \mcC^{\alpha}_\mfm$ is continuous and $\msH_T \subset C_T\mcC^{\alpha}_0(\mbT)$. Consider now $h\in \msH_T$. Then for any $\varepsilon >0$, there exists a $\delta>0$ such that
\begin{equation*}
\|v-h\|_{C_T\mcC^{\alpha}_0} <\delta\,\Rightarrow\,\|u_T-\msS_T(\zeta,h)\|_{\mcC^{\alpha}}<\varepsilon,
\end{equation*} 
and therefore
\begin{equation*}
\mbP(\|u_T-\msS_T(\zeta,h)\|_{\mcC^{\alpha}}<\varepsilon) \geq \mbP(	\|v-h\|_{C_T\mcC^{\alpha}_0} <\delta) >0,
\end{equation*}
where the last inequality follows from Lemma \ref{lem:SHESupport} and this shows \eqref{eq:SolMapClosureInclusion}. It now suffices to show that,
\begin{equation*}
    \overline{\left\{ \msS_T(\zeta,h)\,:\,h\in \msH_T \right \}}^{\|\,\cdot\,\|_{\mcC^{\alpha}}}= \mcC_\mfm^{\alpha}(\mbT).
\end{equation*}
Let $y\in C_{\mfm}^\infty(\mbT)$.
Then for $t\in[0,T]$ define
\begin{equation}
    u^y_t := e^{t\Delta}\zeta + \frac{t}{T}(y-e^{T\Delta}\zeta),
\end{equation}
and
\begin{equation*}
    h^y_t := -\int_0^t e^{(t-s)\Delta} \partial_x((u^y_s)^2\partial_x \rho_{u^y_s}) \dd s
    + \frac{t}{T}(y-e^{T\Delta}\zeta).
\end{equation*}
Since $\zeta\in \mcC^{\alpha}_\mfm$, we have $u^y \in C_T\mcC_\mfm^{\alpha}$ and it also follows that $h^y \in C_T\mcC_0^{\alpha}$ with $h^y(0)=0$.
Furthermore, by construction,
\begin{equation}
    \msS_T(\zeta,h^y)= u_T^y=y.
\end{equation}
Approximating $h^y$ by functions in $C^\infty_0([0,T]\times \mbT)\cap \msH_T$
and using the density of $C^\infty(\mbT)$ in $\mcC^{\alpha}(\mbT)$ concludes the proof.
\end{proof}
\subsection{Exponential Mixing}\label{subsec:ExponentialErgodicity}
In Theorem \ref{th:SemiGroupContract1} and Corollary \ref{cor:1dUniqueInvariant} below, we keep $(\alpha_0,\,\alpha)$ satisfying \eqref{eq:exponentCriteria2}, the more restrictive parameter set from the start of Subsection \ref{subsec:StrongFeller}.
\begin{theorem}\label{th:SemiGroupContract1}
There exists $\lambda \in (0,1)$ such that for every $\zeta,\,\tilde{\zeta} \in \mcC_\mfm^{\alpha_0}$ and $t\geq 1$,
\begin{equation}\label{eq:SemiGroupContract1}
\|P^*_t\delta_\zeta -P^*_t \delta_{\tilde{\zeta}}\|_{\TV} \,\,\leq \,\lambda.
\end{equation}
\end{theorem}
\begin{proof}[Sketch of Proof]
It suffices to essentially repeat the proof of \cite[Thm.~6.5]{tsatsoulis_weber_18}. Given an arbitrary fixed time interval, the idea is to divide the interval into three sub-intervals and repeatedly evolve the solutions, $u(\zeta),\,u(\tilde{\zeta})$, for the duration of each sub-interval into an open bounded set, couple the semi-groups started from these points and successively apply the results of Theorems \ref{thm:SolAPriori}, \ref{th:SemiGroupLocalLipschitz} and \ref{th:SolFullSupport}.
\end{proof}
\begin{corollary}\label{cor:1dUniqueInvariant} For any $\mfm\in \mbR$, there exists a unique measure $\nu \in\mcP( \mcC^{\alpha_0}_{\mfm})$ which is invariant for the semi-group $(P_t)_{t\geq 0}$. Furthermore,
for any $\delta \in (0,1/2)$, $\nu(\mcC^{1/2-\delta}_\mfm)=1$ and $\nu$ has full support in $\mcC^{1/2-\delta}_\mfm$.
Finally, there exists $\lambda\in (0,1)$ such that for every $\mu \in \mcP(\mcC^{\alpha_0}_\mfm)$ and 
$t\geq 1$
\begin{equation}\label{eq:SemiGroupExponential}
\|P^*_t\mu -\nu\|_{\TV}\, \leq \lambda^{\floor{t}}\|\mu-\nu\|_{\TV(\mcC^{\alpha_0}_\mfm)}.
\end{equation}
\end{corollary}
\begin{proof}
Using the existence of invariant measures, shown in Theorem~\ref{thm:existence_invar_measures}, the proof of uniqueness follows exactly as that of \cite[Cor. 6.6]{tsatsoulis_weber_18}.
\end{proof}
We now complete the proof of Theorem \ref{th:1dexpErgodic}.
\begin{proof}[Proof of Theorem \ref{th:1dexpErgodic}] Let $(\alpha_0,\,\alpha,\,\eta)$ satisfy the criteria of \eqref{eq:exponentCriteria}, the less restrictive set of parameters,
and $(\alpha',\,\alpha_0',\,\eta')$ satisfy the more restrictive set, \eqref{eq:exponentCriteria2}, and be such that $\alpha_0' >\alpha_0$.
Then applying Corollary \ref{cor:1dUniqueInvariant} with $(\alpha_0',\,\alpha',\,\eta')$ we find that there exists a unique invariant measure
$\nu\in\mcP(\mcC^{\alpha}_\mfm)$ of $(P_t)_{t\geq 0}$
and that $\nu$ has full support in $\mcC^{\alpha}_\mfm$.
From Theorem \ref{th:1dGWP}, given $\zeta \in \mcC^{\alpha_0}(\mbT)$, for any $t>0$, $u_t(\zeta)\in \mcC^{\alpha}_\mfm (\mbT)\subset \mcC^{\alpha_0'}_{\mfm}(\mbT)$. Therefore, setting $\mu= \mcL(u_t(\zeta))$, in Corollary \ref{cor:1dUniqueInvariant}, we have, for some $\lambda \in (0,1)$,
\begin{align*}
\|P_{2t}^*\delta_\zeta - \nu\|_{\TV(\mcC^{\alpha})}& = \|P^*_{t} \mcL\left(u_{t}(\zeta)\right) -\nu\|_{\TV(\mcC^{\alpha_0})}\\
&\leq \lambda^{\floor{t}} \| \mcL\left(u_{t}(\zeta)\right) -\nu\|_{\TV(\mcC^{\alpha'_0})} \\
&\leq \lambda^{\floor{t}},
\end{align*}
where in the penultimate line we used that $\mcC_b(\mcC^{\alpha_0}(\mbT)) \subset\mcC_b(\mcC^{\alpha_0'}(\mbT))$ and in the last line that the total variation distance is bounded by $1$. Therefore~\eqref{eq:LawConverge} holds for $c>0$ sufficiently small.

Consider now $p\geq 1$.
By Fernique's theorem  (\cite[Thm.~2.6]{daprato_zabczyk_14}),
there exists $\Lambda>0$ such that
$\mbE[\exp(\Lambda \|v\|_{C_1\mcC^\alpha}^2)]<\infty$.
Since $\mcC^\alpha\hookrightarrow L^p$
and $\|u_1(\zeta)\|_{L^p} \lesssim 1+\|v\|_{C_1\mcC^\alpha}^{1/\alpha}$
by Theorem~\ref{th:1dRemainderAPriori},
there exists $\Lambda>0$ such that
\begin{equation}
\sup_{\zeta\in\mcC^{\alpha_0}}
\mbE\left[ \exp\left(\Lambda \|u_1(\zeta)\|^{2\alpha}_{L^p}\right) \right] <\infty\;.
\end{equation}
Hence, taking $\alpha$ arbitrarily close to $\frac12$ (after possibly restarting the equation and choosing a new $\alpha_0$ close to $0$), we obtain~\eqref{eq:LawTail}.
\end{proof}
\appendix
\section{H{\"o}lder-Besov Spaces on \texorpdfstring{$\mbT$}{T}}\label{app:HolderBesovSpaces}
For $\alpha \in \mbR$ and $p,\,q \in [1,\infty]$ we define the non-homogeneous Besov norm for $f\in C^\infty(\mathbb{T})$ by the expression
\begin{equation}\label{eq:NonHomBesovNorm}
\|f\|_{\mcB^\alpha_{p,\,q}(\mbT)} := \left\|\left(2^{\alpha k}\|\Delta_k f\|_{L^p(\mbT)}\right)_{k}\right\|_{l^q(\mathbb{Z})},
\end{equation}
where $\Delta_k : L^1(\mbT)\rightarrow L^1(\mbT)$ are the Littlewood--Paley projection operators on the periodic, integrable functions.
See for example \cite[Ch. 2]{bahouri_chemin_danchin_11} for a construction on $\mbR^d$ and \cite[App. A]{mourrat_weber_17_DFI} for a discussion of the same construction in the periodic setting.
We recall here that there exist kernels $h_k \in C^\infty(\mbT)$ such that, for $f \in L^1(\mbT)$ and $k\geq -1$, $\Delta_k f = h_k \ast f$ - see \cite[Sec. 2.2]{bahouri_chemin_danchin_11}. Furthermore $\mcF h_k$, the Fourier transform of $h_k$, takes values in $[0,1]$ and is supported in the ball $B(0,2)$ for $k=-1$ and in the annulus $B(0,9\cdot2^{k})\setminus B(0,2^{k})$ for $k\geq 0$.

We denote by $\mcB^\alpha_{p,\,q}(\mbT)$ the completion of $C^\infty(\mathbb{T})$ with respect to \eqref{eq:NonHomBesovNorm}, which ensures these spaces are separable.
\begin{theorem}[{Duality Pairing - \cite[Prop.~A.1]{mourrat_weber_17_DFI}}]
\label{th:BesovDuality}
For $\alpha \in\mbR$, $p,\,q\in [1,\infty]$ and $p',q' \in [1,\infty]$ such that $\frac{1}{p}+ \frac{1}{p'} =\frac{1}{q}+\frac{1}{q'}=1$, there exists a constant $C:=C(\alpha,p,q)>0$ such that
\begin{equation}\label{eq:BesovDuality}
|\langle f,g\rangle| \,\leq C \|f\|_{\mcB^\alpha_{p,q}}\|g\|_{\mcB^{-\alpha}_{p',q'}}.
\end{equation}
\end{theorem}
\begin{theorem}[{Besov Embeddings - \cite[Prop.~A.2, Rem.~A.3]{mourrat_weber_17_DFI}}]\label{th:BesovEmbedding}
Let $\alpha \leq \beta \in \mbR$, $q_1,\,q_2 \in [1,\infty]$, $p\geq r \in [1,\infty]$ be such that $\beta = \alpha+\left( \frac{1}{r}-\frac{1}{p}  \right)$ and $q_1\geq q_2$. Then there exists a constant $C:=C(\alpha,p,r,q_1,q_2)>0$ such that
\begin{equation}\label{eq:BesovEmbedding}
\|f\|_{\mcB^\alpha_{p,q_1}}\leq C \|f\|_{\mcB^\beta_{r,q_1}}, \quad\text{and}\quad \|f\|_{\mcB^{\alpha}_{p,q_1}}\leq C \|f\|_{\mcB^{\alpha}_{p,q_2}}.
\end{equation}
Furthermore, for any $\beta>\alpha$ and $p,q\in [1,\infty]$ the embedding $\mcB^\beta_{p,q}\hookrightarrow \mcB^{\alpha}_{p,q}$ is compact.
Finally, there exists $C:=C(p)>0$ such that
\begin{equation}\label{eq:Besov_Lp_embedding}
C^{-1}\|f\|_{\mcB^0_{p,\infty}}\leq
\|f\|_{L^p}\leq \|f\|_{\mcB^0_{p,1}}\;.
\end{equation}
\end{theorem}
\begin{theorem}[{Effect of Derivatives - \cite[Prop.~A.5]{mourrat_weber_17_DFI}}]\label{th:BesovDerivative}
Let $\alpha \in \mbR$, $p,\,q\in [1,\infty]$ and $k \in \mbN_{>0}$. Then there exists  a constant $C:=C(k)>0$ such that for any $\sigma \in \mbN$
\begin{equation}\label{eq:BesovDerivative}
\|D^\sigma f\|_{\mcB^{\alpha-\sigma}_
{p,q}}\leq C \|f\|_{\mcB^{\alpha}_{p,q}}.
\end{equation}
\end{theorem}
\begin{theorem}[{Bounds in Terms of the Derivative - \cite[Prop.~A.6]{mourrat_weber_17_DFI}}]\label{th:BesovPoincare}
Let $\alpha \in (0,1]$ and $p,\,q \in [1,\infty]$, where if $\alpha=1$ we impose $q=\infty$. Then there exists a constant $C:=C(\alpha,p,q)>0$ such that
\begin{equation}\label{eq:BesovPoincare}
\|f\|_{\mcB^{\alpha}_{p,q}} \leq C\left( \|f\|_{L^p}^{1-\alpha}\|\partial_x f\|^\alpha_{L^p} + \|f\|_{L^p}\right)
\end{equation}
\end{theorem}
\begin{theorem}[{Product Bounds - \cite[Prop.~A.7]{mourrat_weber_17_DFI}}]\label{th:BesovProduct}
Let $p,q \in [1,\infty]$, $\alpha,\,\beta \in \mbR$ be such that $\alpha+\beta>0$ with $\beta >0$.
Then for $\alpha>0$ and $\nu \in [0,1]$, there exists a constant $C:=C(\alpha,\beta,p,q)>0$ such that
\begin{equation}\label{eq:BesovProduct}
\|fg\|_{\mcB^{\alpha}_{p,q}} \leq C \|f\|_{\mcB^{\alpha}_{\frac{p}{\nu},q}}\|g\|_{\mcB^{\beta}_{\frac{p}{1-\nu},q}}\leq C \|f\|_{\mcB^{\alpha+(1-\nu)\frac{d}{p}}_{p,q}}\|g\|_{\mcB^{\beta+\nu\frac{d}{p}}_{p,q}}.
\end{equation}
Furthermore, if $\alpha>0$ and $p_1,p_2,p_3,p_4\in [1,\infty]$ are such that $\frac{1}{p_1}+\frac{1}{p_2}=\frac{1}{p_3}+\frac{1}{p_4}=\frac{1}{p}$, then there exists a $C:= C(\alpha,p,q,p_1,p_2,p_3,p_4)>0$ such that
\begin{equation}\label{eq:BesovPositiveProductSplit}
\|fg\|_{\mcB^{\alpha}_{p,q}} \leq C\left(\|f\|_{L^{p_1}}\|g\|_{\mcB^{\alpha}_{p_2,q}} + \|f\|_{\mcB^{\alpha}_{p_3,q}}\|g\|_{L^{p_4}} \right).
\end{equation}
\end{theorem}
\subsection{Parabolic and Elliptic Regularity Estimates}
We define the action of the heat semi-group on $f \in L^1(\mbT)$ by,
\begin{equation*}
e^{t\Delta}f := \mcF^{-1}\left(e^{-4\pi^2|\,\cdot\,|^2t} \hat{f}_{\,\cdot\,}\right) = \mcH_t\ast f, \quad t>0,
\end{equation*}
where $\mcF^{-1}$ denotes the inverse Fourier transform on $L^1(\mbT)$ and
\begin{equation}\label{eq:PeriodicHeatKernel}
\mathcal{H}_t(x) := \frac{1}{\sqrt{4\pi t}}\sum_{n\in \mbZ} e^{-\frac{|x-n|^2}{4t}} \indic_{(0,\infty)}(t)= \sum_{m\in \mbZ} e^{-4\pi^2|m|^2t}e_{m}(x)\indic_{(0,\infty)}(t),
\end{equation}
 defines the heat kernel on $(0,\infty)\times \mbT$. We refer to \cite[Prop.~5 \& 6]{mourrat_weber_17_GWP} for a proof of the following theorem.
\begin{theorem}[{Regularising Effect of the Heat Flow - \cite[Prop.~5 \& 6]{mourrat_weber_17_GWP}}]\label{th:BesovHeatFlow}
Let $\alpha,\,\beta \in \mbR$, $p\geq r \in [1,\infty]$ and $q\in [1,\infty]$. Then, if $\beta\leq \alpha\leq \beta+2$, there exists a constant $C:=C(\alpha,\beta,p,r,q)>0$, such that, uniformly over $t>0$,
\begin{equation}\label{eq:HeatFlow}
\|e^{t\Delta}f\|_{\mcB^\alpha_{p,q}} \leq C t^{-\frac{1}{2}\left(\frac{1}{r} -\frac{1}{p} \right)-\frac{1}{2}\left(\alpha - \beta\right)}\|f\|_{\mcB^{\beta}_{r,q}}.
\end{equation}
Secondly, if $0\leq \beta-\alpha\leq 2$ then there exists a $C:=C(\alpha,\beta,p,q)>0$ such that for a any $t>0$,
\begin{equation}\label{eq:HeatFlowMinusId}
\|(1-e^{t\Delta}) f\|_{\mcB^{\alpha}_{p,q}} \leq C t^{\frac{\beta-\alpha}{2}} \|f\|_{\mcB^{\beta}_{p,q}}.
\end{equation}
\end{theorem}
\begin{remark}\label{rem:HeatSemiGroupContAtZero}
Since we have defined the H{\"o}lder-Besov spaces as the closure of $C^\infty(\mbT)$ under the appropriate norm, for any $\alpha \in \mbR$, $p,\,q \in [1,\infty]$, if $f\in \mcB^{\alpha}_{p,q}(\mbT)$, we have that
$$
\lim_{t\rightarrow 0}\|(1-e^{t\Delta})f\|_{\mcB^{\alpha}_{p,q}} = 0.
$$
\end{remark}
The equivalent elliptic regularity estimate is an easy consequence of \cite[Lem. 2.2]{bahouri_chemin_danchin_11} concerning Fourier multipliers.
\begin{theorem}\label{th:BesovElliptic}
Let $f\in C^\infty(\mbT)$ be such that $\langle f,1\rangle =0$.
Suppose $-\partial_{xx} \rho = f$ and $\langle \rho,1\rangle =0$. Then for any $\alpha\in \mbR$, $p,q\in [1,\infty]$, there exists a constant $C\geq 0$ such that
\begin{equation}\label{eq:BesovElliptic}
\|\rho\|_{\mcB^{\alpha}_{p,q}} \leq C \|f\|_{\mcB^{\alpha-2}_{p,q}}.
\end{equation}
\end{theorem}
\begin{proof}
From \cite[Lem. 2.2]{bahouri_chemin_danchin_11}, there exists a $C>0$ such that for any $k\geq 0$ we have $\|\Delta_k \rho\|_{L^p}  \leq C2^{-2k}\|\Delta_kf\|_{L^p}$.
Since $(\mcF f)(0)=(\mcF\rho)(0)=0$, $\Delta_{-1} \rho(\zeta) = \Delta_{-1}f(\zeta)$. Therefore, there exists a dimension dependent constant $C>0$, such that for any $k\geq -1$,
\begin{equation*}
2^{\alpha k} \|\Delta_k\rho\|_{L^p} \leq C 2^{(\alpha-2)k}\|\Delta_k f \|_{L^p},
\end{equation*}
from which the claim follows.
\end{proof}
\begin{remark}
Applying Theorem \ref{th:BesovElliptic} followed by Theorem \ref{th:BesovDerivative} gives us the estimate, for $f,\rho, \alpha,p,q$ as in Theorem~\ref{th:BesovElliptic} and $r\in [1,p]$,
\begin{equation}\label{eq:BesovGradElliptic}
\|\partial_x \rho \|_{\mcB^{\alpha}_{p,q}} \lesssim \|f\|_{\mcB^{\alpha-1+\left(\frac{1}{r}-\frac{1}{p}\right)}_{r,q}}.
\end{equation}
\end{remark}
\section{Regularity of the Stochastic Heat Equation}\label{app:SHERegularity}
We outline the main steps necessary to prove Theorem \ref{th:MarkovSHESpaceTimeRegular} using the characterisation of H{\"o}lder--Besov spaces. We note that a similar result can be proved using a version of the Garsia--Rodemich--Rumsey lemma, \cite{hu_le_13}. A more detailed presentation of similar results in our context can be found in \cite{mourrat_weber_17_GWP,tsatsoulis_weber_18}. We include a proof directly in our setting for the reader's convenience. A central tool is the following Kolmogorov regularity result which can be found in a similar form as \cite[Lem. 9 \& 10]{mourrat_weber_17_GWP}.
We recall the kernels $h_k \in C^\infty(\mbT)$ from the beginning of Appendix~\ref{app:HolderBesovSpaces}.
\begin{lemma}[{Kolmogorov Regularity - \cite[Lem. 9 \& 10]{mourrat_weber_17_GWP}}]\label{lem:FourierKolmogorov}
Let $(t,\phi) \mapsto Z_t(\phi)$ be a map from $\mbR_+ \times L^2(\mbT)\rightarrow L^2(\Omega,\mcF,\mbP)$ that is linear and continuous in $\phi$ and such that $Z_0=0$.
Assume that for some $p>1$, $\alpha'\in \mbR$, $\kappa'>\frac{1}{p}$ and $T>0$,
there exists a constant $C>0$, such that for all $k\geq -1$, $x\in \mbT$ and $s,\,t \in [0,T]$,
\begin{equation}\label{eq:KolmCondition2}
\mbE\left[ |Z_t(h_k(\,\cdot\,-x)) - Z_s(h_{k}(\,\cdot\,-x))|^p \right]\leq C^p |t-s|^{\kappa' p} 2^{-k\alpha'p}.
\end{equation}
Then, for all $\alpha<\tilde{\alpha}'-\frac{1}{p}$ and $ \kappa \in[0,\kappa'-\frac{1}{p})$,
there exists a modification $\tilde{Z}$ of $Z$ such that
\begin{equation*}
\mbE\left[  \|\tilde{Z}\|^p_{\mcC^\kappa_T\mcB^{\alpha-\kappa}_{p,p}} \right] < \infty .
\end{equation*}
\end{lemma}
\begin{proof}
It suffices to modify the proof of \cite[Lem. 9 \& 10]{mourrat_weber_17_GWP} to our finite volume and one dimensional setting.
\end{proof}
\begin{proof}[Proof of Theorem \ref{th:MarkovSHESpaceTimeRegular}]
Below we show that $\mcL(\mcI_{t_0,t})$ depends only on $|t-t_0|$ so without loss of generality we set $t_0=0$. It is also clear, by translation invariance of $\mcH$ in space and the law of the white noise, that $\mcI_{t_0,t}$ is translation invariant, i.e $\mcI_{t_0,t}(\phi(\,\cdot\,+x)) \sim \mcI_{t_0,t}(\phi(\,\cdot\,))$ for any $\phi \in C^\infty(\mbT)$ and $x\in \mbT$.

For $k\geq -1$, $t \in (0,T]$ and setting $x=0$ in \eqref{eq:KolmCondition2}, we have $v_{0,t}(h_k)=\mcI_{0,t}(h_k)$. So splitting the increment using the semi-group property, Parseval's theorem, the covariance formula for the space-time white noise and the action of the heat kernel in Fourier space we have
for all $0\leq s < t \leq T$
\begin{align*}
	\mbE\left[|\mcI_{0,t}(h_k)-\mcI_{0,s}(h_k)|^2 \right]
	&\lesssim \sum_{\substack{m \in \supp(\mcF h_k)\\m\neq 0}} \frac{1}{8\pi^2|m|^2}(1-e^{-(t-s)8\pi^2 |m|^2})
	\\
	&\lesssim\sum_{\substack{m \in \supp(\mcF h_k)\\m\neq 0}} \frac{|t-s|^\gamma }{(8\pi^2|m|^2)^{1-\gamma}}\\
	&\lesssim
	|t-s|^\gamma 2^{-k(1-2\gamma)},
\end{align*}
where $\mcF h_k$ denotes the Fourier transform of $h_k$ and we may choose any $\gamma \in [0,1]$. So then using Nelson's estimate \cite[Sec. 1.4.3]{nualart_10}, for any $p\geq 2$, there exists a constant $C:=C(p,T)>0$ such that for all $k\geq -1$ and $\gamma \in [0,1)$,
\begin{equation*}
	\mbE\left[|\mcI_{0,t}(h_k)-\mcI_{0,s}(h_k)|^p \right] \leq C |t-s|^{\frac{\gamma}{2}p} 2^{-k\left(\frac{1}{2}-\gamma\right)p}.
\end{equation*}
Then applying Theorem \ref{lem:FourierKolmogorov} and the Besov embedding \eqref{eq:BesovEmbedding}, we see that there exists a modification of $t\mapsto v_{0,t}$ (which we do not relabel) such that for any $p\geq 1$, $\alpha\in (0,1/2),\,\kappa \in \left[0,1/2\right)$ and $T>0$,
\begin{equation*}
    \mbE\left[\|v_{0,\,\cdot\,}\|^p_{\mcC^{\kappa}_T \mcC^{\alpha-2\kappa} }\right] < \infty.
\end{equation*}
To see stationarity of the processes $t_0\mapsto v_{t_0,t_0+h}$ for fixed $h>0$, we use the fact that $v_{t_0,t_0+h}(\phi)$ is Gaussian, and has zero mean, so that its law is entirely determined by its second moment. Therefore, letting $\phi \in C^\infty(\mbT)$, and by similar computations as above, we see that
\begin{align*}
\mbE\left[v_{t_0,t_0+h}(\phi)^2 \right]
&= \sum_m \frac{1}{8\pi^2|m|^2}\left( 1- e^{-8\pi^2|m|^2h}\right)|\hat{\phi}_m|^2,
\end{align*}
and hence $\mcL(v_{t_0,t_0+h})$ depends only on $h\in [t_0,T-t_0]$. Finally, by the assumption that $\xi$ is spatially mean free, for any $t_0\in [0,T)$ and $h\in (T-t_0,T]$, $\mcL(v_{t_0,t_0+h}) \in C^\alpha_0(\mbT)$.
\end{proof}

\bibliographystyle{amsplain}

\end{document}